%% file: cofgenlaxorth.tex
\documentclass[a4paper,reqno]{amsart}

\usepackage{lmodern}
\usepackage{slantsc}
\newcommand{\slkz}{{\normalfont\textsl{\textsc{kz}}}}
\newcommand{\kz}{{\normalfont\textsc{kz}}}
\newcommand{\mathkz}{{\normalfont\kz}}

\usepackage[all,2cell,cmtip]{xy}
\usepackage{mathrsfs,amssymb,stmaryrd,bbm,mathabx}
\xyoption{v2}
\UseAllTwocells
\usepackage{amsmath}
\usepackage{gitinfo}
\usepackage{color}
\definecolor{darkgreen}{rgb}{0,0.45,0}
\usepackage[colorlinks,citecolor=darkgreen,urlcolor=gray,final,backref=page,hyperindex]{hyperref}
\usepackage{autonum}
 \usepackage{tikz}
\usepackage[textsize=scriptsize,color=white]{todonotes}
\usepackage{tensor}
\usepackage[abbrev,msc-links,bibtex-style]{amsrefs}
\usepackage[inline]{enumitem}
\usepackage{colonequals}
\usepackage{mathtools}

\numberwithin{equation}{section}

\usepackage{thmtools,thm-restate}
\declaretheorem[style=plain,numberlike=equation,title=Theorem]{theorem}
\declaretheorem[style=plain,numberlike=equation,title=Corollary]{cor}
\declaretheorem[style=plain,numberlike=equation,title=Lemma]{lemma}
\declaretheorem[style=plain,numberlike=equation,title=Proposition]{prop}

\declaretheorem[style=remark,numberlike=equation,title=Remark]{rmk}

\declaretheorem[style=remark,numberlike=equation,title=Example]{ex}

\declaretheorem[style=remark,numberlike=equation,title=Notation]{notation}
\declaretheorem[style=remark,numberlike=equation,title=Assumption]{Ass}

\declaretheorem[style=definition,numberlike=equation,title=Definition]{df}

\newcommand{\arepmult}{\mathbf{ARep}\multicat}
\newcommand{\clopfib}{\ensuremath{\mathbf{OpFib}}}

\newcommand{\multicat}{\ensuremath{\mathbf{MCat}}}

\newcommand{\VCat}{\ensuremath{\V\text-\mathbf{Cat}}}
\newcommand{\VCAT}{\ensuremath{\V\text-\mathbf{CAT}}}
\newcommand{\Cat}{\ensuremath{\mathbf{Cat}}}
\newcommand{\CAT}{\ensuremath{\mathbf{CAT}}}
\newcommand{\C}{\ensuremath{\mathcal{C}}}
\newcommand{\A}{\ensuremath{\mathcal{A}}}

\DeclareMathOperator{\dom}{dom}
\DeclareMathOperator{\cod}{cod}

\newcommand{\two}{{\mathbf{2}}}

\newcommand{\Sq}{\mathbb{S}\mathrm{q}}

\DeclareMathOperator{\col}{col}
\newcommand{\sperp}{{\pitchfork\hspace{-4.9pt}\pitchfork}}
\newcommand{\uc}{{\circ}}

\DeclareMathOperator{\Ran}{Ran}
\newcommand{\Topo}{\ensuremath{\mathbf{Top}_0}}
\newcommand{\Pord}{\ensuremath{\mathbf{Ord}}}
\newcommand{\Ord}{\ensuremath{\mathbf{Ord}}}

\newcommand{\V}{\ensuremath{\mathcal{V}}}

\DeclareMathOperator{\colim}{colim}

\begin{document}
\title{Cofibrantly generated lax orthogonal factorisation systems}
\author{Ignacio L\'opez~Franco}
\address{Departamento de Matem\'atica y Aplicaciones, CURE, Universidad de la
  Rep\'ublica\\
Tacuaremb\'o s/n, Maldonado, Uruguay}
\email{ilopez@cure.edu.uy}
\thanks{
  The author gratefully acknowledges the support of the following institutions
  during the long gestation of this article: a Research Fellowship of
  Gonville and Caius College, Cambridge; the Department of Pure
  Mathematics and Mathematical Statistics of the University of Cambridge;
  \textsc{sni-anii}, \textsc{pedeciba} and Universidad de la Rep\'ublica.
}
\begin{abstract}
  The present note has three aims. First, to complement the theory of
  cofibrant generation of algebraic weak factorisation systems (\textsc{awfs}s)
  to cover some important examples that are not
  locally presentable categories. Secondly, to prove that cofibrantly
  \textsc{kz}-generated \textsc{awfs}s (a notion we define) are always lax
  orthogonal. Thirdly, to show that
  the two known methods of building lax orthogonal \textsc{awfs}s, namely
  cofibrantly \textsc{kz}-generation and the method of ``simple adjunctions'',
  construct different \textsc{awfs}s. 
  We study in some detail the example of cofibrant \kz-generation that yields
  representable multicategories, and a counterexample to cofibrant generation
  provided by continuous lattices.
\end{abstract}
\maketitle
\section{Introduction}
\label{sec:introduction}

Lax orthogonal factorisation systems (\textsc{lofs}s) are a type of algebraic
weak factorisation systems (\textsc{awfs}s) on 2-categories, introduced in
\cite{Clementino2016458}, for which the diagonal fillers satisfy a universal
property, similar to the property of a left Kan extension. Several examples of
\textsc{lofs}s were constructed in \cite{Clementino2016458,lmcs3960} using the
method of ``simple adjunctions.'' The present note addresses an aspect of this
theory that has not been touched upon: the cofibrant generation of
\textsc{lofs}.
The cofibrant generation of \textsc{awfs}s on locally presentable categories were
studied in~\cite{MR3393453}, encompassing a large family of examples that,
however, do not reach some important ones, as those based on the category of
topological spaces. The present note is an attempt to fill this gap in the
literature.

The article can be divided in two parts. The first, taking most of the
article, that deals with cofibrant generation of
\textsc{awfs}s enriched over a base category $\V\subseteq\Cat$, and the second
that looks at the case when \V\ is the category of preorders.

\textsc{Awfs}s
were introduced by M.~Grandis and W.~Tholen in~\cite{MR2283020}, with latter
contributions by R.~Garner~\cite{MR2506256}, and, as the name indicates, they are an
algebraisation of the more classical notion of a weak factorisation system
(\textsc{wfs}s). A \textsc{wfs} consists of two classes of morphisms
$(\mathcal{L},\mathcal{R})$ with the property that each morphism $f$ can be
written as $f=r\cdot \ell$, with $\ell\in\mathcal{L}$ and $r\in\mathcal{R}$, and
each $r\in \mathcal{R}$ precisely when it has
the right lifting property with respect to each $\ell\in\mathcal{L}$, ie for each
commutative square as displayed, there exists a --~non necessarily unique~--
diagonal filler
\begin{equation}
  \label{eq:12}
  \diagram
  \cdot\ar[r]\ar[d]_\ell&\cdot\ar[d]^r\\
  \cdot\ar@{..>}[ur]\ar[r]&\cdot
  \enddiagram
\end{equation}
and dually, $\ell\in\mathcal{L}$ precisely when it has the left lifting property
with respect to each $r\in\mathcal{R}$.
This kind of lifting situation was common in algebraic topology long
before the importance of \textsc{wfs}s was fully realised, for
which Quillen's definition of model category was central.

One of the features that distinguishes \textsc{awfs}s from
\textsc{wfs}s --~on a category \C, say~-- is that the factorisation of
morphisms is functorial --~as is also the case in many \textsc{wfs}s constructed by
the so-called Quillen's small object argument~\cite{MR0223432}.
The left and right classes of morphisms are replaced, respectively, by the
coalgebras and algebras for a certain comonad and monad on $\C^\two$, and this
extra algebraic structure ensures that diagonal
fillers~\eqref{eq:12} not only exist but can be constructed from algebraic
data.

Lax orthogonal factorisation systems, introduced in~\cite{Clementino2016458}, are
\textsc{awfs}s on 2-categories for which the canonical diagonal, say $d$,
satisfies a certain universal property with respect to 2-cells: given any other
diagonal filler $w$ and 2-cells $\alpha\colon h\Rightarrow w\cdot \ell$ and
$\beta\colon k\Rightarrow r\cdot w$, there exists a unique 2-cell $\gamma\colon d\Rightarrow w$
such that $\gamma\cdot \ell=\alpha$ and $r\cdot\gamma=\beta$.
\begin{equation}
  \label{eq:19}
  \xymatrixcolsep{1.7cm}
  \diagram
  \cdot\ar[r]^-h\ar[d]_\ell&
  \cdot\ar[d]^r\\
  \cdot\urtwocell^d_w{\hole\exists !}\ar[r]&
  \cdot
  \enddiagram
\end{equation}
See
Section~\ref{sec:lax-orth-fact} or~\cite{Clementino2016458} for more on the
definition of \textsc{lofs}, along with the basic theory of \textsc{lofs}s and a
procedure to construct them via the so-called \emph{simple adjunctions}.
When the codomain of $r$ is a terminal object,
the above condition says that the identity 2-cell $h=d\cdot \ell$ exhibits $d$ as a
left Kan extension of $h$ along $\ell$. Objects $A$ with this property with
respect to a family of morphisms $\ell$ have been studied in the context of
poset-enriched categories and called \emph{Kan objects}~\cite{MR1641443} or
\emph{Kan injective objects}~\cite{MR3283679}.

The notion of cofibrant generation adapted to \textsc{awfs}s was introduced
in~\cite{MR2506256}, and latter extended to generation by a double category
in~\cite{MR3393453}, where, in addition, enriched cofibrant generation is discussed.

There is a notion of cofibrant generation for
\textsc{lofs}s, that we call \emph{cofibrant \slkz-generation}. We show that
cofibrantly \kz-generated \textsc{awfs}s are always \textsc{lofs}.
Cofibrant \kz-generation can be
seen as a case of the constructions in~\cite[\S 8]{MR3393453} --~even though
\cite{MR3393453} does not consider \textsc{lofs}s and concentrates on locally
presentable categories. Representable multicategories provide an example that we
study in some detail.
There are
important examples,
however, that are not locally presentable categories, as the
category of topological spaces. We study the case of categories enriched in
posets and the cofibrant \kz-generation thereon, encompassing in this way the
example of topological spaces.

We have mentioned two ways of constructing new \textsc{lofs}s:
via simple adjunctions and via cofibrant \kz-generation. It is natural to ask
whether these two procedures construct the same \textsc{lofs}s. We give a negative
answer to this question, exhibiting a \textsc{lofs} that can be
constructed via simple adjunctions but are not cofibrantly \kz-generated, nor
cofibrantly generated; furthermore, its underlying \textsc{wfs} is not
cofibrantly generated, in the usual sense of the term. The \textsc{lofs} in
question is defined on the category of $T_0$ topological
spaces, its left morphisms are the subspace
embeddings, and its fibrant
objects are the continuous lattices \cite[\S 12]{Clementino2016458}
\cite{lmcs3960,1702.02602}.

We conclude this introduction with a description of the article's
contents. Section~\ref{sec:backgr-algebr-weak} collects some of the
constructions and results relative to \textsc{awfs}s needed in the rest of the
article. In Section~\ref{sec:cofibrant-generation} we adapt the notions of
cofibrant generation introduced in~\cite{MR2506256,MR3393453} to the case of
categories enriched in $\V\subseteq\Cat$.
Lax orthogonal factorisation systems are recalled in
Section~\ref{sec:lax-orth-fact} and the
notion of cofibrant \kz-generation is introduced in Section~\ref{sec:cofibrant-generation-1}.
Before showing that cofibrantly \kz-generated \textsc{awfs}s are \textsc{lofs}s
in Section~\ref{sec:lax-orthogonality-kz}, we show in
Section~\ref{sec:exist-locally-pres} the existence of cofibrantly \kz-generated
\textsc{awfs}s in the locally presentable case.
Section \ref{sec:repr-cat-enrich} looks to the example of representable
multicategories as arising from a cofibrantly \kz-generated \textsc{lofs} on
multicategories.
Section~\ref{sec:cons-cofibr-gener} shows that cofibrantly generated
and \kz-generated \textsc{awfs}s must satisfy certain accessibility, or
colimit-creation property, to be used in a latter section.
Section~\ref{sec:cofibr-kz-gener} proves an existence result for
cofibrantly \kz-generated \textsc{awfs} on preorder-enriched
categories and compares it with \cite{MR3283679}, while
Section~\ref{sec:continuous-lattices-1} exhibits an example of a \textsc{lofs}
on $\mathbf{Top}_0$ that is \emph{not} cofibrantly \kz-generated, nor cofibrantly
generated, and whose underlying \textsc{wfs} is \emph{not} cofibrantly generated
in the usual meaning of the term.
There is an Appendix~\ref{sec:lax-idempotent-2} with background on lax
idempotent 2-monads and our own results on reflections of lax idempotent 2-monads
along functors.

\section{Background on algebraic weak factorisation systems}
\label{sec:backgr-algebr-weak}

As mentioned in the introduction, we are interested in 2-categories and locally
ordered categories. One could choose to develop the exposition in the context of
$\mathcal{V}$-enriched categories, for a fairly general symmetric
monoidal closed category $\mathcal{V}$. On the other hand, one could treat only
the case of 2-categories and, when necessary, argue that the relevant
constructions and results restrict to locally ordered 2-categories. We will
take a middle of the road approach and consider categories enriched in a full
sub-2-category $\mathcal{V}\subseteq\Cat$, closed under limits and exponentials,
and that is cocomplete. Furthermore, it will be important for
our applications that the
arrow category $\two$ should belong to \V. Our main examples of \V\ will be the 2-categories
\Cat\ of small categories and \Pord\
of posets. We do not assume that the inclusion $\V\subseteq\Cat$ preserves
colimits. This, together with $\two\in\V$, would imply that the inclusion is
an equivalence.
Many of the $\mathcal{V}$-enriched notions
we discuss below hold for a much more general $\mathcal{V}$, and are not difficult to
elaborate in that context. Since they do not add much to our main examples, we
leave them to be
developed elsewhere.

\subsection{Functorial factorisations}
\label{sec:funct-fact}

A \V-functorial factorisation on a \V-category $\mathcal{C}$ is a \V-functor
$\C^\two\to\C^{\mathbf{3}}$ that is a section of the composition $\V$-functor
$\C^{\mathbf{3}}\to\C^\two$. Equivalently, it is a \V-functor
$K\colon\C^\two\to\C$ with \V-natural transformations $\dom\Rightarrow
K\Rightarrow\cod$ whose composition equals the canonical transformation
$\dom\Rightarrow\cod$ with $f$-component equal to $f$.
In other words, a $\V$-functorial factorisation is
a functorial factorisation as defined in~\cite{MR2283020} that is compatible
with the 2-cells of \C.
\begin{equation}
  \label{eq:24}
  \big(
  A\xrightarrow{f}B
  \big)
  =
  \big(
  A\xrightarrow{Lf}Kf\xrightarrow{Rf} B
  \big)
\end{equation}
As in the
case of ordinary categories, a functorial factorisation can be equivalently
described by a copointed endo-\V-functor $\Phi\colon L\Rightarrow 1$ on
$\C^\two$ with $\dom\Phi=1$, or by a pointed endo-\V-functor $\Lambda\colon1\Rightarrow R$ on
$\C^\two$ with $\cod\Lambda=1$.

Given a \V-functorial factorisation as in the previous paragraph, each coalgebra
structure $(1,s)\colon f\to Lf$ for the copointed endo-\V-functor $(L,\Phi)$ on
$f\in\C^\two$ and each algebra structure $(p,1)\colon Rf\to f$ for the pointed
endo-\V-functor $(R,\Lambda)$ on $g\in\C^\two$ induces a choice of diagonal
fillers for morphisms $(h,k)\colon f\to g$ in $\C^\two$, ie commutative squares
in \C. The diagonal filler for this
square is the composite
\begin{equation}
  \label{eq:9}
  \operatorname{diag}(h,k)\colon
  \cod f=\dom k\xrightarrow{s}Kf \xrightarrow{K(h,k)} Kg \xrightarrow{p} \cod h
  = \dom g.
\end{equation}
If $(\alpha,\beta)\colon (h,k)\Rightarrow(\bar h,\bar k)\colon f\to g$ is a
2-cell in $\C^\two$, then there is a corresponding 2-cell
$\operatorname{diag}(\alpha,\beta)\colon \operatorname{diag}(h,k)\Rightarrow
\operatorname{diag}(\bar h,\bar k)$, given by
$\operatorname{diag}(\alpha,\beta)=p\cdot K(\alpha,\beta)\cdot s$.
See~\cites{MR2283020,MR2506256}, and \cite{Clementino2016458} for the
2-categorical case.

\subsection{Categories internal to \V-categories}
\label{sec:categ-intern-v}

A category internal to $\mathbf{CAT}$ is a double category, and can be described
as having objects and two kinds of morphisms, horizontal and vertical
ones. Objects together with, respectively, horizontal and vertical morphisms form
a category. Furthermore, there are \emph{squares}, each one of which has a
vertical domain and codomain and a horizontal domain and codomain, and can be
composed both vertically and horizontally. There are other details, like
identity squares and various compatibilities between identities and composition,
that can be found in~\cite{elementsof2cat}.

If, instead of categories internal to $\mathbf{CAT}$, we consider categories
internal to $\VCAT$, the resulting structure is that of a double category except
that now there are 2-cells between horizontal morphisms; objects, horizontal
morphisms and these 2-cells form a \V-category. Furthermore, there are 2-cells
between squares which, together with vertical morphisms as objects and squares
as morphisms, form a \V-category. Each category internal to $\VCAT$ has an
underlying double category obtained by disregarding all the 2-cells.

Most of the categories internal to $\VCAT$ we consider will be related to the
one described below.
The arrow category $\two$ has a cocategory
structure, in the sense that it is an internal category in
$\Cat^{\mathrm{op}}$. Another way of putting this is to say that for each
category $\mathcal{A}$ the set $\Cat(\two,\A)=\operatorname{Mor}\mathcal{A}$
carries a category structure, namely, that of $\mathcal{A}$.

As a consequence of the cocategory structure on $\two$, for any $\V$-category
$\C$ there is a category internal to \VCat, called the \emph{internal category
  of squares} in \C, and depicted on the right below.
\begin{equation}
  \label{eq:3}
  \diagram
  \mathbf{3}\ar@{<-}@<-5pt>[r] \ar@{<-}@<5pt>[r] \ar@{<-}[r]&
  \two\ar[r]&
  \mathbf{1}\ar@<5pt>[l]\ar@<-5pt>[l]
  \enddiagram
  \qquad
  \diagram
  \C^{\mathbf{3}}\ar@<-5pt>[r] \ar@<5pt>[r] \ar[r]&
  \C^\two\ar@{<-}[r]&
  \C\ar@{<-}@<5pt>[l]\ar@{<-}@<-5pt>[l]
  \enddiagram
\end{equation}
The underlying double category of $\Sq(\C)$ has objects those of $\C$ and both
vertical and horizontal morphisms the morphisms of \C. Squares are commutative
squares in \C\ and 2-cells between horizontal morphisms are just 2-cells in
\C. A 2-cell between squares is a pair of 2-cells as depicted, that satisfy
$g\cdot\alpha=\beta\cdot f$.
\begin{equation}
  \label{eq:126}
  \diagram
  \cdot\ar[d]_f\ar[r]^h&\cdot\ar[d]^g\\
  \cdot\ar[r]_k&\cdot
  \enddiagram
  \Longrightarrow
  \diagram
  \cdot\ar[d]_f\ar[r]^{\bar h}&\cdot\ar[d]^g\\
  \cdot\ar[r]_{\bar k}&\cdot
  \enddiagram
  \qquad
  \text{is}
  \qquad
  \diagram
  \cdot\ar[d]_f\rtwocell^h_{\bar h}{\alpha}&\cdot\ar[d]^g\\
  \cdot\rtwocell^k_{\bar k}{\beta}&\cdot
  \enddiagram
\end{equation}

There is an obvious notion of internal functor between categories internal to
\VCAT, the formulation of which is left to the reader.

\subsection{Enriched {{\sc awfs}}{\normalfont s}}
\label{sec:enriched-textscawfss}

A \V-enriched \textsc{awfs} on a \V-category \C\ is a \V-functorial
factorisation on \C\ equipped with a comultiplication $\Sigma\colon L\Rightarrow
L^2$ that makes $\mathsf{L}=(L,\Phi,\Sigma)$ a \V-enriched comonad, and a
multiplication $\Pi\colon R^2\Rightarrow R$ that makes
$\mathsf{R}=(R,\Lambda,\Pi)$ a \V-enriched monad. Furthermore,
that the underlying ordinary comonad and monad on the underlying category of
$\C^\two$ should form an \textsc{awfs} as defined in~\cite{MR2506256}; in other words,
the \V-natural transformation $LR\Rightarrow RL$ with components
$(\sigma_f,\pi_f)\colon LRf\to RLf$ must be a mixed distributive law; here we
have used the notation $\Sigma_f=(1,\sigma_f)\colon Lf\to L^2f$ and
$\Pi_f=(\pi_f,1)\colon R^2f\to Rf$.
This distributivity condition, added in~\cite{MR2506256} to the original
definition of \textsc{awfs} --~called natural \textsc{wfs}s
in~\cite{MR2283020}~-- is precisely what is needed in order to have an
associative composition of $\mathsf{R}$-algebras and of
$\mathsf{L}$-coalgebras. An associative composition of $\mathsf{R}$-algebras
chooses, for each pair of $\mathsf{R}$-algebras $f$, $g$ with $\cod f=\dom g$,
an $\mathsf{R}$-algebra structure on the composition $g\cdot f$; these
assignments must be natural with respect to morphisms and 2-cells of
$\mathsf{R}$-algebras, and it must be associative in the sense that the
$\mathsf{R}$-algebra structures of $(h\cdot g)\cdot f$ and $h\cdot (g\cdot f)$
must coincide. A similar statement can be made about $\mathsf{L}$-coalgebras.

The existence of the associative composition mentioned in the previous paragraph
can be rephrased in a more concise way:
if $(\mathsf{L},\mathsf{R})$ is a \V-enriched \textsc{awfs}, there are internal
categories in \VCAT\
\begin{equation}
  \label{eq:11}
  \diagram
  \mathsf{L}\text-\mathrm{Coalg}\ar@<4pt>[r]\ar@{<-}[r]\ar@<-4pt>[r]&
  \C
  \enddiagram
  \qquad
  \diagram
  \mathsf{R}\text-\mathrm{Alg}\ar@<4pt>[r]\ar@{<-}[r]\ar@<-4pt>[r]&
  \C
  \enddiagram
\end{equation}
that we denote by $\mathsf{L}\text-\mathbb{C}\mathrm{oalg}$ and
$\mathsf{R}\text-\mathbb{A}\mathrm{lg}$. These two internal categories come
equipped with internal functors into $\Sq(\C)$ given by forgetting the
(co)algebra structure. Furthermore, given a \V-monad $\mathsf{R}$ on
\C, there is a bijection between compositions that make
$\mathsf{R}\text-\mathrm{Alg}\rightrightarrows\C$ into an internal category in
\VCAT\ and \V-comonads $\mathsf{L}$ such that $(\mathsf{L},\mathsf{R})$ is a
\V-enriched \textsc{awfs}.
This is completely analogous to the
case of ordinary \textsc{awfs}s \cite[\S 2.8]{MR3393453}.

The internal categories in \VCAT\ that arise from an \textsc{awfs} can be
characterised as in~\cite[\S 3]{MR3393453}. If
$\mathbb{D}=(\mathcal{D}_1\rightrightarrows\mathcal{D}_0)$ is an internal
category and $U\colon \mathbb{D}\to\Sq(\C)$ an internal functor in \VCAT, then
$\mathbb{D}$ is isomorphic to $\mathsf{R}\text-\mathbb{A}\mathrm{lg}$ over $\Sq(\C)$, for
an --~essentially unique~-- \V-enriched \textsc{awfs} $(\mathsf{L},\mathsf{R})$
on \C, if the following two conditions hold:
\begin{enumerate*}[label=(\alph*)]
\item
$U_1\colon\mathcal{D}_1\to\C^\two$ is monadic (it has a \V-enriched left adjoint
and the comparison \V-functor into the algebras of the associated \V-monad is an
\emph{isomorphism});
\item the induced
\V-monad is isomorphic to a codomain-preserving \V-monad.
\end{enumerate*}
A more elementary
condition can be found in~\cite[Thm.~6]{MR3393453}.

\subsection{Internal categories of {\sc lari}{\normalfont s} and {\sc
    rali}{\normalfont s}}
\label{sec:intern-categ-laris}
  Later we will need to refer to certain internal \V-categories whose vertical
  morphisms are given by adjunctions.

  A morphism $f\colon A\to B$, in a \V-category $\mathcal{A}$, equipped with a right adjoint whose
  unit is an identity 2-cell may be called a \emph{left adjoint right inverse}, abbreviated
  \textsc{lari}, following terminology used
  in~\cite{MR0213413,MR3393453}. Thus, a \textsc{lari} is a triple
  $(f,r,\varepsilon)$, where $f\colon A\to B$ and $r\colon B\to A$ are
  morphisms and $\varepsilon\colon f\cdot r\Rightarrow 1_B$ is a
  2-cell, satisfying
  $ r\cdot f=1_A$, $\varepsilon\cdot f=1$, $r\cdot\varepsilon =1$.
  A morphism
  of \textsc{lari}s is a morphism between the underlying morphisms in \C\ that commutes with the
  right adjoints and the
  corresponding counits; explicitly, a morphism
  $(f,r,\varepsilon)\to(f',r',\varepsilon')$ is a morphism
  $(h,k)\colon f\to f'$ in $\mathcal{A}^\two$ that satisfies $r'\cdot
  k=h\cdot r$ and $k\cdot\varepsilon=\varepsilon\cdot k$. Finally, a 2-cell
  between two morphisms $(h,k)\Rightarrow(\bar h,\bar k)$ is simply a 2-cell in
  $\mathcal{A}^\two$.
  We have, for any \V-category $\mathcal{A}$, a \V-category
  $\mathbf{Lari}(\mathcal{A})$ with objects \textsc{lari}s in $\mathcal{A}$ and
  morphisms and 2-cells as described above.

  Clearly, \textsc{lari}s compose, in the sense that if $f\colon A\to B$ and
  $g\colon B\to C$ are morphisms with \textsc{lari} structures
  $(f,r,\varepsilon)$ and $(g,s,\varphi)$, then $(g\cdot f,r\cdot
  s,\varepsilon(r\cdot\varphi\cdot f))$ is a \textsc{lari}. Furthermore, identity
  morphisms have an obvious \textsc{lari} structure. As a consequence there is an internal
  category  $\mathbf{Lari}(\mathcal{A})\rightrightarrows\mathcal{A}$ in \VCAT,
  denoted by $\mathbb{L}\mathbf{ari}(\mathcal{A})$, and an obvious internal
  functor $\mathbb{L}\mathbf{ari}(\mathcal{A})\to\Sq(\mathcal{A})$.

  Dually, a morphism $f\colon A\to B$ in a \V-category $\mathcal{A}$ equipped
  with a left adjoint whose unit is an identity may be called a \emph{right adjoint
  left inverse}, or \textsc{rali}. A morphism of \textsc{rali}s is a morphism in
  $\mathcal{A}^\two$ between the underlying morphisms that commutes with the left
  adjoints and the units. There is a internal category
  $\mathbb{R}\mathbf{ali}(\mathcal{A})$ in \VCAT\ and a forgetful internal
  functor $\mathbb{R}\mathbf{ali}(\mathcal{A})\to\Sq(\mathcal{A})$.

  Each one of the internal categories $\mathbb{L}\mathbf{ari}(\mathcal{A})$ and
  $\mathbb{R}\mathbf{ali}(\mathcal{A})$ over $\Sq(\mathcal{A})$ are induced by
  an
  \textsc{awfs} under weak assumptions on the \V-category $\mathcal{A}$. Recall
  that, as part of our assumptions, the category $\V\subseteq\Cat$ contains the
  arrow category $\two$ and is closed under limits so it makes sense to speak of
  comma-objects in \V; these comma-objects are just comma-categories. The lax limit of a morphism $f\colon X\to Y$ in \V\ is another
  name for the comma-category $f/Y$. A \emph{lax limit} of a morphism $f\colon A\to B$ in
  a \V-category $\mathcal{A}$ is a diagram as the one depicted on the left, that
  is sent by each representable $\mathcal{A}(C,-)$ to a lax limit in \V.
  A \emph{lax colimit} of $f$ is a diagram as shown on the right that is sent by each
  $\mathcal{A}(-,C)$ to a lax limit in \V.
  \begin{equation}
    \label{eq:15}
    \xymatrixrowsep{.1cm}
    \xymatrixcolsep{1.8cm}
    \diagram
    &A\ar[dd]^f\\
    f/B\ar[ur]\ar[dr]\rtwocell<\omit>&\\
    &B
    \enddiagram
    \qquad
    \diagram
    A\ar[dd]_f\ar[dr]&\\
    {}\rtwocell<\omit>{^}&\col_\ell f\\
    B\ar[ur]&
    \enddiagram
  \end{equation}

  If $\mathcal{A}$ has lax limits of morphisms, then
  $\mathbb{L}\mathbf{ari}(\mathcal{A})\cong
  \mathsf{L}\text-\mathbb{C}\mathrm{oalg}$, where $(\mathsf{L},\mathsf{R})$ are
  given on a morphism $f\colon A\to B$ in the following way: $L(f)\colon A\to
  f/B$ is the
  morphism that corresponds to the identity 2-cell $f\Rightarrow f$, and $R(f)\colon
  f/B\to B$ is the projection; see~\cite[\S 3.3]{Clementino2016458}.

  If $\mathcal{A}$ has lax colimits of morphisms, then there is a \V-enriched \textsc{awfs}
  $(\mathsf{E},\mathsf{M})$ on $\mathcal{A}$ such that
  $\mathbb{R}\mathbf{ali}(\mathcal{A})\cong\mathsf{M}\text-\mathbb{A}\mathrm{lg}$,
  where $(\mathsf{E},\mathsf{M})$ is given on a morphism $f\colon A\to B$ in the
  following way: $E(f)\colon A\to \col_\ell f$ is the coprojection and
  $M(f)\colon\col_\ell f\to B$ is the morphism corresponding to the identity
  2-cell $f\Rightarrow f$.

\section{Categories of lifting operations}
\label{sec:cofibrant-generation}

Classically, a \textsc{wfs}
$(\mathcal{L},\mathcal{R})$ on a category $\mathcal{C}$ is cofibrantly generated
by a set of morphisms $\mathcal{J}$ if $\mathcal{R}$ is the family of morphisms
that have the right lifting operation with respect to each member of
$\mathcal{J}$. Quillen's small object argument proves that each set of morphisms
$\mathcal{J}$ cofibrantly generates a \textsc{wfs} in categories
that satisfy certain smallness and cocompleteness conditions. In the context of
\textsc{awfs}s, right lifting properties are replaced by lifting operations,
whose definition we recall below, and the set of morphisms $\mathcal{J}$ can be
replaced by a category $\mathcal{J}$ and a functor
$U\colon\mathcal{J}\to\mathcal{C}^{\two}$.

A previous version of this manuscript used a slightly different
exposition of enriched \textsc{awfs}s from that of \cite{MR3393453}. As
our main concern here is to study the cofibrant \kz-generation of \textsc{awfs}s
instead of providing an overarching exposition of enriched \textsc{awfs}s, we
shall follow \cite[\S 8]{MR3393453} and save space. The only difference in
our, admittedly short, exposition is the fact that
all the \textsc{awfs}s will be $\V$-enriched.


A lifting operation against $U$ on $f$
can be described as a family
\begin{equation}
  \label{eq:33}
  \mathcal{C}^\two(Uj,f)\longrightarrow\mathcal{C}(\cod Uj,\dom f),
\end{equation}
natural in $j\in\mathcal{J}$, each of which is a section to the dashed morphism
into the pullback displayed below.
\begin{equation}
  \label{eq:13}
  \diagram
  \mathcal{C}(\cod Uj,\dom f)\ar@{-->}[r]\ar@/^15pt/[rr]^{\mathcal{C}(Uj,1)}
  \ar@/_/[dr]_{\mathcal{C}(1,f)}
  &\mathcal{C}^\two (Uj,f)\ar[r]\ar[d]\ar@{}[dr]|{\mathrm{pb}}&
  \mathcal{C}(\dom Uj,\dom f)\ar[d]^{\mathcal{C}(1,f)}\\
  &\mathcal{C}(\cod Uj,\cod f)\ar[r]_-{\mathcal{C}(Uj,1)}&
  \mathcal{C}(\dom Uj,\cod f)
  \enddiagram
\end{equation}
Morphisms of $\mathcal{C}$ equipped with a lifting operation against $U$ form a
category $\mathcal{J}^\pitchfork$ that has an obvious forgetful functor
$U^\pitchfork\colon\mathcal{J}^\pitchfork\to\mathcal{C}^\two$ that forgets the
lifting operation.

In order to put these ideas in a setting that allows for an easy generalisation
to the enriched case and takes account of the double category structures, we
shall follow~\cite{MR3393453} in the utilisation of fibre squares. If
$\mathcal{A}$ is a \V-category with enriched pullbacks, there is a \V-functor $C\colon
\mathcal{A}^{\two\times\two}\to\mathcal{A}^\two$ that sends a square to the
comparison morphism into the associated pullback square.
\begin{equation}
  \label{eq:80}
  \xymatrixcolsep{1.2cm}
  \xymatrixrowsep{1.2cm}
  \diagram
  \cdot\ar[r]^h\ar[d]_f&\cdot\ar[d]^g\\
  \cdot\ar[r]^k&\cdot
  \enddiagram
  \longmapsto
  \xymatrixcolsep{.6cm}
  \xymatrixrowsep{.6cm}
  \diagram
  \cdot\ar@{..>}^{C(f,g,h,k)}[dr]&&\\
  &\cdot\ar[r]\ar[d]\ar@{}[dr]|{\mathrm{pb}}&\cdot\ar[d]^g\\
  &\cdot\ar[r]^k&\cdot
  \enddiagram
\end{equation}
If $\mathcal{A}$ carries an enriched \textsc{awfs} $(\mathsf{E},\mathsf{M})$,
then the $\V$-category of fibre squares is defined by the following
pullback.
\begin{equation}
  \label{eq:133}
  \diagram
  \operatorname{Fibre}_{\mathsf{M}}(\mathcal{A})
  \ar[d]\ar[r]&
  \mathsf{M}\text-\mathrm{Alg}\ar[d]\\
  \mathcal{A}^{\two\times\two}\ar[r]^-C&
  \mathcal{A}^\two
  \enddiagram
\end{equation}
Its objects are squares as in~\eqref{eq:80} with an $\mathsf{M}$-algebra
structure on $C(f,g,h,k)$, called \emph{fibre squares}. It is not hard to show
that if we ``paste'' two of these squares, or set them side by side, the resulting outer
square is also a fibre square (the proof depends on the fact that the codomain
functor $\mathsf{M}\text-\mathrm{Alg}\to\mathcal{A}$ is a discrete pullback-fibration). A symmetry
argument ensures that two fibre squares, one stacked on top of the other, yields
a new fibre square. 
By the
aforementioned, there are two internal category structures on
$\mathrm{Fibre}_{\mathsf{M}}(\mathcal{A})$, one given by pasting squares side by
side, and the
other by pasting squares vertically.


We now assume that \V\ is equipped with a \V-enriched \textsc{awfs}
$(\mathsf{E},\mathsf{M})$. Any $\V$-functor category
$[\mathcal{J}^{\mathrm{op}},\V]$ has an induced \textsc{awfs}, defined
pointwise, that we still call $(\mathsf{E},\mathsf{M})$.

Given a \V-functor $U\colon \mathcal{J}\to\C^\two$, there is an induced
\V-functor
$\hat U\colon
\mathcal{C}^\two\to[\mathcal{J}^{\mathrm{op}},\V]^{\two\times\two}$ that sends
$f$ to the outer square in~\eqref{eq:13} (where $j$ is a variable). This is in
fact an internal functor
$\hat U\colon\Sq(C)\to\Sq([\mathcal{J}^{\mathrm{op}},\V]^\two)$.  The
\V-category $\mathcal{J}^{\pitchfork_{\mathsf{M}}}$ is defined by the pullback
diagram on the left below. The \V-functor $W_U=\hat U C$ sends $f\colon A\to B$
to the dashed transformation in~\eqref{eq:13}. An object of
$\mathcal{J}^{\pitchfork_{M}}$ is a morphism $f\colon A\to B$ with an
$\mathsf{M}$-algebra structure on each dashed arrow as in~\eqref{eq:13} that
vary $\V$-naturally with $j\in \mathcal{J}$.
\begin{equation}
  \label{eq:131}
  \diagram
  \mathcal{J}^{\pitchfork_{\mathsf{M}}}
  \ar[d]_{U^{\pitchfork_{\mathsf{M}}}}\ar[r]
  &
  \mathrm{Fibre}_{\mathsf{M}}([\mathcal{J}^{\mathrm{op}},\V])
  \ar[d]\ar[r]
  &
  \mathsf{M}\text-\mathrm{Alg}
  \ar[d]
  \\
  \C^\two
  \ar[r]^-{\hat U}
  &
  [\mathcal{J}^{\mathrm{op}},\V]^{\two\times\two}
  \ar[r]^-C
  &
  [\mathcal{J}^{\mathrm{op}},\V]^\two
  \enddiagram
  \qquad
  \diagram
  \mathcal{J}^{\pitchfork_{\mathsf{M}}}\ar[d]\ar[r]&
  \mathbb{F}\mathrm{ibre}_{\mathsf{M}}[\mathcal{J}^{\mathrm{op}},\V]\ar[d]\\
  \Sq(\C)\ar[r]&
  \Sq([\mathcal{J}^{\mathrm{op}},\V]^\two)
  \enddiagram
\end{equation}
Since $\hat U$ is part of an internal functor $\Sq(\C)\to
\Sq([\mathcal{J}^{\mathrm{op}},\V]^\two)$, it is immediate that
$\mathcal{J}^{\pitchfork_{\mathsf{M}}}$ is the arrow part of an internal
category, with object part $\C$, and that it fits in a pullback square of internal
functors displayed on the right above.

\begin{ex}
  \label{rmk:3}
  Suppose that $\V=\mathbf{Set}$ and that $\mathsf{M}$ is the free split
  epimorphism monad on $\mathbf{Set}^{\two}$, which is explicitly given by sending
  $f\colon A\to B$ to $\binom{f}{1} \colon A+B\to B$.
  A lifting operation on a morphism $g\colon A\to B$ in a category $\mathcal{C}$
  against a functor $U\colon \mathcal{J}\to\mathcal{C}^\two$ is a choice of
  section $\phi_j$ for each canonical $\mathcal{C}(\cod
  Uj,A)\to\mathcal{C}^\two(Uj,g)$ in such a way that the family $\phi_j$ is
  natural in $j$.
  \begin{equation}
    \label{eq:54}
    \diagram
    \dom Uj\ar[d]_{Uj}\ar[r]^-h&
    A\ar[d]^g\\
    \cod Uj\ar[r]_-k \ar@{..>}[ur]|{\phi_j(h,k)}&
    B
    \enddiagram
  \end{equation}
  For this particular monad $\mathsf{M}$, we will suppress the suffix in the
  notation $\mathcal{J}^{\pitchfork_{\mathsf{M}}}$ and simply write
  $\mathcal{J}^{\pitchfork}$, as to coincide with the notation used
  in~\cite{MR2506256}.
\end{ex}

Assume that $\mathcal{J}$ is part of an internal category in \VCat, say
$\mathbb{J}=(\mathcal{J}\rightrightarrows \mathcal{J}_0)$, with an internal
functor $(U,U_0)\colon \mathbb{J}\to\Sq(\C)$.
We shall now describe the internal category
$\mathbb{J}^{\sperp_{\mathsf{M}}}$ over $\Sq(\C)$.
There are two internal functors $\mathcal{J}^{\pitchfork_{\mathsf{M}}}\to
(\mathcal{J}\times_{\mathcal{J}_0}\mathcal{J})^{\pitchfork_{\mathsf{M}}}$, which
at the level of object of objects are the identity
$\mathcal{J}_0\to\mathcal{J}_0$. At the level of object of arrows, one of these
internal functors corresponds to the \V-functor
$\mathcal{J}^{\pitchfork_{\mathsf{M}}}\to
\mathrm{Fibre}_{\mathsf{M}}[(\mathcal{J}\times_{\mathcal{J}_0}\mathcal{J})^{\mathrm{op}},\V]$
that sends $f$ to the pasted fibre square below. Here the fibre square structure is
the one given by the pasting of fibre squares. The second internal functor from
$\mathcal{J}^{\pitchfork_{\mathsf{M}}}$ to
$(\mathcal{J}\times_{\mathcal{J}_0}\mathcal{J})^{\pitchfork_{\mathsf{M}}}$
corresponds to the \V\ functor
$\mathcal{J}^{\pitchfork_{\mathsf{M}}}\to
\mathrm{Fibre}_{\mathsf{M}}[(\mathcal{J}\times_{\mathcal{J}_0}\mathcal{J})^{\mathrm{op}},\V]$
that sends $f$ to the fibre square of the outer rectangle.
\begin{equation}
  \label{eq:4}
  \xymatrixcolsep{1.5cm}
  \xymatrixrowsep{.5cm}
  \diagram
  \C(\cod Uj,A)
  \ar[d]_{\C(1,f)}\ar[r]^{\C(Uj,1)}&
  \C(\cod Ui,A)
  \ar[r]^{\C(Ui,1)}\ar[d]^{\C(1,f)}&
  \C(\dom Ui,A)
  \ar[d]^{\C(1,f)}\\
  \C(\cod Uj,B)\ar[r]^{\C(Uj,1)}&
  \C(\cod Ui,B)\ar[r]^{\C(Ui,1)}&
  \C(\dom Ui,B)
  \enddiagram
\end{equation}
The internal category $\mathbb{J}^{\sperp_{\mathsf{M}}}$ is the equaliser of
this pair of internal functors.
\begin{equation}
  \label{eq:5}
  \mathbb{J}^{\sperp_{\mathsf{M}}}\to \mathcal{J}^{\pitchfork_{\mathsf{M}}}
  \rightrightarrows
  (\mathcal{J}\times_{\mathcal{J}_0}\mathcal{J})^{\pitchfork_{\mathsf{M}}}
\end{equation}

Perhaps the easiest way to understand $\mathbb{J}^{\sperp_{\mathsf{M}}}$ is to have a
quick look at $\mathbb{J}^{\sperp}$, that is, the case when $\mathsf{M}$ is
the monad whose algebras are split epimorphisms.
\begin{ex}
  \label{ex:12}
If $\mathbb{J}$ is a double category, with underlying graph
$\mathcal{J}\rightrightarrows \mathcal{J}_0$, and $\mathbb{U}=(U,U_0)\colon
\mathbb{J}\to\Sq(\mathcal{C})$ is a double functor, we may consider
$\mathcal{J}^\pitchfork$, whose objects are morphisms $g$ of $\mathcal{C}$
equipped with a lifting operation $\phi$ against $U$, as explicitly described in
Example~\ref{rmk:3}. It is natural to ask for the following extra condition: if
$i$ and $j$ are composable vertical morphisms of $\mathbb{J}$, then
$\phi_j(\phi_i(h,k\cdot Uj),k)$ must be equal to $\phi_{j\bullet i}(h,k)$, the
diagonal filler given by the lifting operation against the vertical composition
of $i$ and $j$. This condition always holds in orthogonal
factorisation systems, ie when the diagonal fillers are unique, but it is in
general untrue.
\begin{equation}
  \label{eq:55}
  \xymatrixrowsep{.4cm}
  \xymatrixcolsep{4.5cm}
  \diagram
  \cdot\ar[d]_{Ui}\ar[r]^-h&
  \cdot\ar[dd]^g\\
  \cdot\ar[d]_{Uj}
  \ar@{..>}[ur]|(.35){\phi_i(h,k\cdot Uj)}
  &\\
  \cdot\ar[r]_-k
  \ar@{..>}[uur]|(.4){\phi_j(\phi_i(h,k\cdot Uj),k)}
  &
  \cdot
  \enddiagram
  \qquad
  \xymatrixcolsep{3cm}
  \diagram
  \cdot\ar[d]_{Ui}\ar[r]^-h&
  \cdot\ar[dd]^g\\
  \cdot\ar[d]_{Uj}
  &\\
  \cdot\ar[r]_-k
  \ar@{..>}[uur]|(.5){\phi_{j\bullet i}(h,k)}
  &
  \cdot
  \enddiagram
\end{equation}
The category $\mathbb{J}^\sperp$ is the full subcategory of
$\mathcal{J}^\pitchfork$ whose objects satisfy the compatibility condition
described above.
\end{ex}

\begin{ex}
  \label{ex:13}
  Suppose given a functorial factorisation on a category $\C$, with associated
  domain-preserving copointed endofunctor $(L,\Phi)$ and associated
  codomain-preserving pointed endofunctor $(R,\Lambda)$. Then, each
  $(R,\Lambda)$-algebra $(p,1)\colon Rg\to g$ induces an object $(g,\phi)$ of
  $(L,\Phi)\text-\mathrm{Coalg}^\pitchfork$ in the following way. If
  $(1,s)\colon f\to Lf$ is an $(L,\Phi)$-coalgebra, then
  \begin{equation}
    \label{eq:127}
    {\xymatrixcolsep{2cm}
      \xymatrixrowsep{1.4cm}
      \diagram
      {\cdot}
      \ar[r]^-{h}\ar[d]_{f}
      &
      {\cdot}
      \ar[d]^{g}
      \\
      {\cdot}
      \ar[r]_-{k}\ar@{..>}[ur]|{\phi_{(g,s)}(h,k)}
      &
      {\cdot}
      \enddiagram}
    \quad=\quad
    \xymatrixrowsep{.6cm}
    \xymatrixcolsep{2cm}
    \diagram
    \cdot\ar[dd]_f\ar@{=}[r]&
    \cdot\ar[d]|{Lf}\ar[r]^h&
    \cdot\ar[d]|{Lg}\ar@{=}[r]&
    \cdot\ar[dd]^g
    \\
    &\cdot\ar[r]^-{K(h,k)}\ar[d]|{Rf}&
    \cdot\ar[d]|{Rg}\ar[ur]_p&
    \\
    \cdot\ar[ur]^s\ar@{=}[r]&
    \cdot\ar[r]_-k&
    \cdot\ar@{=}[r]&
    \cdot
    \enddiagram
  \end{equation}
  This assignment can easily seen to be part of a functor over $\C^\two$
  \begin{equation}
    \label{eq:93}
    (R,\Lambda)\text-\mathrm{Alg}\longrightarrow (L,\Phi)\text-\mathrm{Coalg}^{\pitchfork}.
  \end{equation}

  In the case when $(R,\Lambda)$ and $(\mathsf{L},\Phi)$ underlie an
  \textsc{awfs} $(\mathsf{L},\mathsf{R})$, the inclusions
  $\mathsf{R}\text-\mathrm{Alg}\hookrightarrow (R,\Lambda)\text-\mathrm{Alg}$
  and $\mathsf{L}\text-\mathrm{Coalg}\hookrightarrow
  (L,\Phi)\text-\mathrm{Coalg}$ induce, together with~\eqref{eq:93}, a functor
  over $\C^\two$
  $ \mathsf{R}\text-\mathrm{Alg}\hookrightarrow
    (R,\Lambda)\text-\mathrm{Alg}\longrightarrow
    (L,\Phi)\text-\mathrm{Coalg}^{\pitchfork}
    \longrightarrow \mathsf{L}\text-\mathrm{Coalg}^{\pitchfork}$
  that is simply a restriction of the lifting operation described above.
\end{ex}

One can say a few words regarding the behaviour of $U^{\pitchfork_{\mathsf{M}}}$
and (co)limits.
\begin{lemma}
  \label{l:20}
  Suppose given a pullback diagram of \V-functors as displayed.
  \begin{equation}
    \xymatrix@R=.3cm@C=.8cm{
      \mathcal{P}\ar[d]_P\ar[r]\ar@{}[dr]&
      \mathcal{B}\ar[d]^G\\
      \mathcal{E}\ar[r]^-F&
      \mathcal{F}
    }
  \end{equation}
  Suppose that $D\colon \mathcal{D}\to\mathcal{P}$ is a \V-functor and that $PD$
  has a (co)limit weighted by the weight $\varphi$. Then $P$ creates this
  (co)limit if $F$ preserves it and $G$ creates (co)limits
  of functors with domain $\mathcal{D}$ weighted by $\varphi$.
\end{lemma}
Since $U^{\pitchfork_{\mathsf{M}}}$ is the pullback of the monadic
$\mathsf{M}\text-\mathrm{Alg}\to [\mathcal{J}^{\mathrm{op}},\V]^\two$ along
$W_U\colon\C^\two\to[\mathcal{J}^{\mathrm{op}},\V]^\two$, we readily obtain:
\begin{lemma}
  \label{l:14}
  Let $U\colon\mathcal{J}\to\C^\two$ be a \V-functor with $\mathcal{J}$ small. Then:
  \begin{enumerate*}
  \item \label{item:23} $W_U$ preserves any limit that may exist in $\C^\two$
    and it has a left adjoint if $\C$ is cocomplete.
  \item \label{item:16} $U^{\pitchfork_{\mathsf{M}}}$ creates limits.
  \item \label{item:19} ${U}^{\pitchfork_{\mathsf{M}}}$ creates
    $U^{\pitchfork_\mathsf{M}}$-split coequalisers.
  \end{enumerate*}
\end{lemma}%

Since, up to isomorphism, the object part of the internal functor
$U^{\sperp_{\mathsf{M}}}\colon\mathbb{J}^{\sperp_{\mathsf{M}}}\to\Sq(\C)$ is
the identity, we will use the same notation for the arrow part
$U^{\sperp_{\mathsf{M}}}\colon\mathbb{J}^{\sperp_{\mathsf{M}}}\to\C^\two$.
\begin{lemma}
  \label{l:17}
  \label{cor:5}
  The \V-functor $U^{\sperp_{\mathsf{M}}}$ enjoys the following properties.
  \begin{enumerate}
  \item It creates limits and $U^{\sperp_{\mathsf{M}}}$-split coequalisers.
  \item It is monadic if it has a left adjoint.
    If the base \V-category $\C$ has cotensor products with $\two$,
    then it suffices that the underlying ordinary functor of
    $U^{\sperp_{\mathsf{M}}}$ should have a left adjoint.
  \end{enumerate}
\end{lemma}
\begin{proof}
  The enriched version of Beck's monadicity theorem together Lemma~\ref{l:17}
  mean that monadicity is guaranteed by the existence of a \V-enriched left
  adjoint. If $\C$ has a cotensor products with $\two$, then
  $U^{\sperp_{\mathsf{M}}}$ creates them, and therefore
  $\mathbb{J}^{\sperp_{\mathsf{M}}}$ has and $U^{\sperp_{\mathsf{M}}}$ preserves
  cotensor products with $\two$. In these circumstances, a left adjoint for the
  underlying ordinary functor induces an enriched left adjoint;
  see~\cite[Prop.~3.1]{BKP}.
\end{proof}

One of the observations of~\cite{MR3393453} is that, for any \textsc{awfs}
$(\mathsf{L},\mathsf{R})$, the functor
$\mathsf{R}\text-\mathrm{Alg}\to\mathsf{L}\text-\mathrm{Coalg}^\pitchfork$
introduced in~\cite{MR2506256} that
expresses the fact that each $\mathsf{R}$ algebra has a canonical lifting
operation against $\mathsf{L}$-coalgebras, co-restricts to an isomorphism
\begin{equation}
  \label{eq:67}
  \mathsf{R}\text-\mathrm{Alg}\cong\mathsf{L}\text-\mathrm{Coalg}^\sperp.
\end{equation}
The existence of this isomorphism can easily be extended from ordinary categories
to 2-categories, or in our case, to categories enriched in our base of
enrichment $\V\subseteq\Cat$ satisfying our blanket hypotheses.




\section{Lax orthogonal factorisation systems}
\label{sec:lax-orth-fact}

This section is a short exposition of the basic definitions of lax orthogonal
\textsc{awfs}s of~\cite{Clementino2016458}, beginning with lax idempotent 2-monads
--~which are closely related to Kock--Z\"oberlein doctrines~\cite{Kock:KZmonads,MR0424896}.

Lax orthogonal \textsc{awfs}s were introduced in~\cite{Clementino2016458}, in
the context of 2-categories. The modification to our framework of categories
enriched in a category $\V\subseteq\Cat$ that satisfies our blanket conditions
is trivial.
\begin{df}
  \label{df:1}
  A \emph{lax orthogonal factorisation system} (\textsc{lofs}) on \C\ is a
  \V-enriched \textsc{awfs} $(\mathsf{L},\mathsf{R})$ whose 2-comonad
  $\mathsf{L}$ and 2-monad $\mathsf{R}$ are lax idempotent. Equivalently, as
  shown in~\cite[\S 4]{Clementino2016458}, either $\mathsf{L}$ or $\mathsf{R}$
  should be lax idempotent.
\end{df}
\begin{ex}
  \label{ex:3}
  The two \textsc{awfs}s of \S \ref{sec:intern-categ-laris} are \textsc{lofs}s.
\end{ex}
\begin{Ass}
  For the rest of the section we equip \V\ with the \textsc{lofs}
  $(\mathsf{E},\mathsf{M})$ described in \S \ref{sec:intern-categ-laris}
  --~so $\mathsf{M}$
  is the free \textsc{rali} \V-monad.
\end{Ass}

Lax orthogonality is closely related to the notion of \kz-lifting
operation~\cite[\S 5]{Clementino2016458}. If $f$ and $g$ are morphisms in \C, a
\kz-lifting operation from $f$ to $g$ is a \textsc{rali} structure on the dashed morphism
induced by the universal property of pullbacks.
If a \kz-lifting operation from $f$ to $g$ exists we say that \emph{$f$ and
  $g$ are \slkz-orthogonal}.
\begin{equation}
  \label{eq:23}
  \diagram
  \C(\cod f,\dom g)\ar@/^15pt/[rr]^{\C(f,1)} \ar@/_/[dr]_{\C(1,g)}
  \ar@{-->}[r]
  &\C^\two(f,g)\ar[r]\ar[d]\ar@{}[dr]|{\mathrm{pb}}&
  \C(\dom f,\dom g)\ar[d]^{\C(1,g)}\\
  &\C(\cod f,\cod g)\ar[r]^-{\C(f,1)}&
  \C(\dom f,\cod g)
  \enddiagram
\end{equation}

Slightly more generally, a \kz-lifting operation from a \V-functor $U\colon
\mathcal{A}\to\C^\two$ to another $V\colon\mathcal{B}\to\C^\two$ is a \textsc{rali}
structure on the comparison \V-natural transformation $\C(\cod U,\dom
V)\Rightarrow \C^\two(U,V)$ defined in a similar fashion to the previous
paragraph. In terms of diagonal fillers, a \kz-lifting operation from $U$ to $V$
can be described as an assignment of a diagonal filler $d(h,k)$ for each
morphism $(h,k)\colon Ua\to Vb$ in $\C^\two$, that is, for each commutative
square of the form depicted.
\begin{equation}
  \label{eq:65}
  \diagram
  \cdot\ar[d]_{Ua}\ar[r]^h&
  \cdot\ar[d]^{Vb}\\
  \cdot\ar@{..>}[ur]|{d(h,k)}\ar[r]_k&
  \cdot
  \enddiagram
\end{equation}
These diagonals must be \V-natural in $a\in\mathcal{A}$ and
$b\in\mathcal{B}$. Furthermore, if $e$ is a morphism parallel to $d(h,k)$, there
must be a bijection between the following two sets of 2-cells:
\begin{itemize}
\item 2-cells $\gamma\colon d(h,k)\Rightarrow e$; and
\item pairs of 2-cells $\alpha\colon h\Rightarrow e\cdot Ua$ and $\beta\colon
  k\Rightarrow Vb\cdot e$ such that $Vb\cdot\alpha=\beta\cdot Ua$.
\end{itemize}
The bijection must be given by $\gamma\mapsto (\gamma\cdot Ua,Vb\cdot
\gamma)$. More details can be found in~\cite{Clementino2016458}.
Clearly, \kz-lifting operations are unique up to isomorphism. 


\begin{theorem}[{\cite[Thm.~6.6]{Clementino2016458}}]
  \label{thm:6}
  A \V-enriched {\normalfont\textsl{\textsc{awfs}}} $(\mathsf{L},\mathsf{R})$ on
  \C\ is a {\normalfont\textsl{\textsc{lofs}}} if and only if there exists a
  \slkz-lifting operation from the
  forgetful $U\colon\mathsf{L}\text-\mathrm{Coalg}\to\C^\two$ to the forgetful
  $V\colon\mathsf{R}\text-\mathrm{Alg}\to\C^\two$.
\end{theorem}

\begin{notation}
  \label{not:2}
  When $\mathcal{V}$ is equipped with the \textsc{awfs}
  $(\mathsf{E},\mathsf{M})$ whose $\mathsf{M}$-algebras are the \textsc{rali}s
  in \V --~see \S \ref{sec:intern-categ-laris}~-- then $\mathbb{J}^{\sperp_{\mathsf{M}}}$ and
  $\mathcal{J}^{\pitchfork_{\mathsf{M}}}$ will be denoted by
  $\mathbb{J}^{\sperp_{\mathkz}}$ and $\mathcal{J}^{\pitchfork_{\mathkz}}$, respectively. The latter
  fits in the following pullback square.
  and is the universal \V-category over $\C^\two$ equipped with a \kz-lifting
  operation from $U$ to $U^{\pitchfork_{\mathkz}}$ --~see~\cite[\S
  6]{Clementino2016458}.
  \begin{equation}
    \label{eq:26}
    \xymatrixrowsep{.5cm}
    \diagram
    \mathcal{J}^{\pitchfork_{\mathkz}}\ar[d]_{U^{\pitchfork_{\mathkz}}}\ar[r]
    &
    \mathbf{Rali}[\mathcal{J}^{\mathrm{op}},\V]\ar[d]
    \\
    \C^{\two}\ar[r]^-{W_U}
    &
    [\mathcal{J}^{\mathrm{op}},\V]^\two
    \enddiagram
  \end{equation}
\end{notation}

\begin{theorem}
  \label{thm:7}
  A \V-enriched {\normalfont\textsl{\textsc{awfs}}} $(\mathsf{L},\mathsf{R})$ on a \V-category \C\ is
  lax orthogonal if and only if\/
  $\mathsf{R}\text-\mathbb{A}\mathrm{lg}\cong\mathsf{L}\text-\mathbb{C}\mathrm{oalg}^{\sperp_{\mathkz}}$
  over $\C^\two$.
\end{theorem}
\begin{proof}
  We use above notations.
  First observe that, for any \V-category $\mathcal{A}$ over $\C^\two$ the
  square displayed on the left is a pullback because \textsc{rali}s are split epis.
  Moreover, the all four $\V$-functors are full and faithful and injective on objects.
  \begin{equation}
    \label{eq:1}
    \xymatrixrowsep{.5cm}
    \diagram
    \mathcal{A}^{\sperp_{\mathkz}}\ar[r]\ar[d]&
    \mathcal{A}^{\pitchfork_{\mathkz}}\ar[d]\\
    \mathcal{A}^{\sperp}\ar[r]&
    \mathcal{A}^\pitchfork
    \enddiagram
    \qquad
    \diagram
    \mathsf{R}\text-\mathrm{Alg}\ar[r]^-{\mathrm{ff}}\ar[d]_\cong&
    \mathsf{L}\text-\mathrm{Coalg}^{\pitchfork_{\mathkz}}\ar[d]^{\mathrm{ff}}\\
    \mathsf{L}\text-\mathrm{Coalg}^{\sperp}\ar[r]^-{\mathrm{ff}}&
    \mathsf{L}\text-\mathrm{Coalg}^{\pitchfork}
    \enddiagram
  \end{equation}
  Theorem~6.6 of~\cite{Clementino2016458} proves that our \textsc{awfs} is lax
  orthogonal if and only if there exists a \V-functor
  $\mathsf{R}\text-\mathrm{Alg}\to\mathsf{L}\text-\mathrm{Coalg}^{\pitchfork_{\mathkz}}$,
  commuting with the forgetful functors into $\C^\two$; this \V-functor is, moreover, full
  and faithful.  Then, if the \textsc{awfs} is lax orthogonal, there is a
  commutative square displayed on the right above, where the isomorphism
  $\mathsf{R}\text-\mathrm{Alg}\cong\mathsf{L}\text-\mathrm{Coalg}^\sperp$ over
  $\C^\two$ is that mentioned in~\eqref{eq:67}. The fact that this diagram is a
  pullback square will follow from the following general argument.

  A commutative square of full and faithful \V-functors is a pullback square if
  and only if it is a pullback at the level of objects. Any diagram of injective
  functions between sets
  \begin{equation}
    \diagram
    \cdot\ar[d]_\cong\ar@{>->}[r]&\cdot\ar@{>->}[d]\\
    \cdot\ar@{>->}[r]&\cdot
    \enddiagram
  \end{equation}
  is a pullback square if the marked function is a bijection. These two facts
  show that the diagram on the right of~\eqref{eq:1} is a pullback.

  For the converse, an isomorphism
  $\mathsf{R}\text-\mathbb{A}\mathrm{lg}\cong\mathsf{L}\text-\mathbb{C}\mathrm{oalg}^{\sperp_{\mathkz}}$
  together with the inclusion
  $\mathsf{L}\text-\mathrm{Coalg}^{\sperp_{\mathkz}}\to\mathsf{L}\text-\mathrm{Coalg}^{\pitchfork_{\mathkz}}$
  gives a \V-functor
  $\mathsf{R}\text-\mathrm{Alg}\to\mathsf{L}\text-\mathrm{Coalg}^{\pitchfork_{\mathkz}}$,
  so the \textsc{awfs} is lax orthogonal by \cite[Thm.~6.6]{Clementino2016458}, concluding
  the proof.
\end{proof}

\section{Cofibrant KZ-generation}
\label{sec:cofibrant-generation-1}

A \textsc{wfs} $(\mathcal{L},\mathcal{R})$ on a category $\mathcal{C}$ is said
to be cofibrantly generated by a family of morphisms
$\mathcal{J}\subset\operatorname{Mor}\mathcal{C}$ if $\mathcal{R}$ consists of those
morphisms with the right lifting property against all members of
$\mathcal{J}$. Certain cocompleteness and smallness conditions on $\mathcal{C}$
guarantee that any small set of morphisms $\mathcal{J}$ cofibrantly generates a
--~unique~-- \textsc{wfs}. This result is usually called Quillen's small object
argument.
A notion of cofibrant generation for \textsc{awfs}s and the
corresponding small object argument appeared in~\cite{MR2506256} and was
later
built upon in~\cite{MR3393453}. The latter article also discussed cofibrant
generation of \textsc{awfs}s on enriched categories, of which we can now say a
few words. Our base of enrichment $\mathcal{V}\subseteq\Cat$ will be equipped with an
\textsc{awfs} $(\mathsf{E},\mathsf{M})$.

\begin{df}
\label{df:8}
Let \C\ be a \V-category, $\mathbb{J}$ an internal category in $\VCat$ and
$\mathbb{U}\colon \mathbb{J}\to\Sq(\C)$ an internal functor.
A \V-enriched \textsc{awfs} $(\mathsf{L},\mathsf{R})$ on $\C$ is
\emph{cofibrantly
  $(\mathsf{E},\mathsf{M})$-generated} by
$\mathbb{J}$ when there is an isomorphism of internal categories
$\mathsf{R}\text-\mathbb{A}\mathrm{lg}\cong \mathbb{J}^{\sperp_{\mathsf{M}}}$ over $\Sq(\C)$.
This is a straightforward modification of~\cite[\S 6.2]{MR3393453}.
\end{df}
The \V-enriched \textsc{awfs} cofibrantly generated by an internal category
$\mathbb{J}$ in \VCat\ exists if and only if $\mathbb{J}^{\sperp_{\mathsf{M}}}\to\C^\two$ is
monadic. This can be shown by modifying~\cite[Thm.~6]{MR3393453}
to the case of \V-categories.

\begin{df}
\label{df:9}
We say that $(\mathsf{L},\mathsf{R})$ is cofibrantly generated \emph{by a
  \V-category} $\mathcal{J}$ over $\C^\two$ if there is an isomorphism of
internal categories
$\mathsf{R}\text-\mathrm{Alg}\cong\mathcal{J}^{\pitchfork_{\mathsf{M}}}$ over
$\Sq(\C)$.
\end{df}

\begin{ex}
  \label{ex:11}
  When $\mathsf{M}$ is the \V-monad on $\V^\two$ given by
  ${M}(f)=1_{\cod(f)}$, then any \textsc{awfs} that is cofibrantly
  $(\mathsf{E},\mathsf{M})$-generated is orthogonal.
\end{ex}

\begin{df}
  \label{df:3}
The notion of cofibrant \textsc{kz}-generation of a \V-enriched \textsc{awfs} is
a special case of Definition~\ref{df:8} when $\V$ is equipped with the
\textsc{awfs} $(\mathsf{E},\mathsf{M})$, described in
\S \ref{sec:intern-categ-laris}, for which $\mathsf{M}$-algebras are
\textsc{rali}s.
\end{df}

\begin{ex}
  \label{ex:4}
  The \textsc{awfs} $(\mathsf{E},\mathsf{M})$ on \V\ whose $\mathsf{M}$-algebras
  are \textsc{rali}s is cofibrantly \kz-generated by the discrete \V-category
  $ \bigl(\mathbf{0}\to\mathbf{1}\bigr)\subset\V^\two$.
  So, $\mathbb{R}\mathbf{ali}(\Cat)\cong
  \bigl(\mathbf{0}\to\mathbf{1}\bigr)^{\pitchfork_{\mathkz}}$.
  Observe that $\mathbb{R}\mathbf{ali}(\Cat)$ is not of the form
  $\mathcal{J}^\pitchfork$ for a small category
  $\mathcal{J}$~\cite[Prop.~17]{MR3393453} but can be obtained as
  $\mathbb{J}^\pitchfork$ for a small double category
  $\mathbb{J}$~\cite[Prop.~19]{MR3393453}.
\end{ex}
\begin{ex}[Opfibrations]
  \label{ex:7}
  There is an isomorphism over \Cat
  \begin{equation}
    \label{eq:88}
    (\mathbf{1}\xrightarrow{0}\two)^{\pitchfork_{\mathkz}} \cong
    \mathbf{Opfib}_s
  \end{equation}
  with the 2-category of \emph{cloven opfibrations} and morphisms in $\C^\two$
  that strictly preserve the cleavages. Indeed, an object of the left hand side
  of~\eqref{eq:88} is a functor $g\colon A\to B$ for which
  \begin{equation}
    \label{eq:89}
    A^\two=\Cat(\two,A)\longrightarrow \Cat^\two(0,g)=g\downarrow 1_B \qquad
    (a\xrightarrow{\alpha}a')\mapsto(a,g(\alpha))
  \end{equation}
  has a \textsc{rali} structure. This structure is well-known to be equivalent
  to an opfibration structure on $g$. The opcartesian lifting of a morphism
  $\beta\colon g(a)\to b$ in $B$ is the diagonal filler $\phi(a,\beta)$ of the
  displayed square.
  The 2-monad of the induced \textsc{awfs} on \Cat\ is the free cloven
  opfibration 2-monad.
  \begin{equation}
    \diagram
    \mathbf{1}\ar[r]^-a\ar[d]_0&A\ar[d]^g\\
    \two\ar[r]_-\beta\ar@{..>}[ur]|{\phi(a,\beta)}&B
    \enddiagram
  \end{equation}
\end{ex}
\begin{ex}[Normal opfibrations]
  \label{ex:6}
  The 2-category of normal cloven opfibrations, by which we mean that the
  opcartesian lifting of an identity morphism is an identity morphism, can be
  cofibrantly \kz-generated by the subcategory of $\Cat^\two$ with objects
  $0\colon \mathbf{1}\to\two$ and $1\colon \mathbf{1}\to \mathbf{1}$, and a
  single non-identity morphism.
  \begin{equation}
    \label{eq:90}
    \diagram
    \mathbf{1}\ar[d]_0\ar[r]&\mathbf{1}\ar[d]\\
    \two\ar[r]&\mathbf{1}
    \enddiagram
  \end{equation}
\end{ex}
\begin{ex}[Split opfibrations]
  \label{ex:8}
  The 2-category of split opfibrations can be obtained by adding to the previous
  example the \kz-generating cofibration
  \begin{equation}
    \label{eq:92}
    \two=(0\to 1)\xrightarrow{\phantom{m}\delta_2\phantom{m}}(0\to 1\to 2)=\mathbf{3}
  \end{equation}
  that is the inclusion that misses out the object $2\in\mathbf{3}$.
  We require compatibility with the following squares, so the 2-category of
  split opfibrations is $\mathbb{J}^{\sperp_{\mathkz}}$, where $\mathbb{J}$ is
  the internal subcategory of $\Sq(\Cat)$ generated by the square of
  Example~\ref{ex:6} and
  the displayed squares --~where $\delta_i$ is the order-preserving
  inclusion that misses $i$.
  \begin{equation}
    \label{eq:91}
    \xymatrixcolsep{1.6cm}
    \diagram
    \mathbf{1}
    \ar[r]^-{1=\delta_0}
    \ar[d]_{0=\delta_1}
    &
    \two\ar[d]^{\delta_2}\\
    \two
    \ar[r]^-{\delta_0}
    &
    \mathbf{3}
    \enddiagram
    \qquad
    \xymatrixrowsep{.5cm}
    \diagram
    \mathbf{1}\ar@{=}[r]\ar[dd]_{\delta_1}&
    \mathbf{1}\ar[d]^{\delta_1}\\
    &\two\ar[d]^{\delta_2}\\
    \two\ar[r]^-{\delta_1}&
    \mathbf{3}
    \enddiagram
  \end{equation}
  If $g\colon A\to B$ is a functor, as seen in previous examples, a lifting
  operation against $\delta_1\colon\mathbf{1}\to\mathbf{2}$ is a choice of an
  opcartesian lifting $\bar\beta$ for each $\beta\colon g(a)\to b$ in $B$. A
  lifting operation against $\delta_2\colon\two\to\mathbf{3}$ is a choice, for each morphism $\alpha$ in
  $A$ and each morphism $\beta'$ in $B$ with $\dom\beta'=\cod g(\alpha)$, of an
  opcartesian lifting for $\beta'$. The compatibility with the square on the
  left above is the requirement that this opcartesian lifting should be equal to
  $\bar\beta$.
  \begin{equation}
    \label{eq:125}
    \diagram
    A\ar[d]_g&
    \cdot\ar[r]^\alpha&
    \cdot \ar@{..>}[r]^{\bar{\beta}'}&
    \cdot
    \\
    B&
    \cdot\ar[r]^\beta&\cdot\ar[r]^{\beta'}&\cdot
    \enddiagram
    \qquad
    \diagram
    \cdot\ar[r]^{\bar\beta}\ar@/_/[rr]_{\widebar{{\beta}'\cdot\beta}}&
    \cdot \ar@{..>}[r]^{\bar{\beta}'}&
    \cdot
    \\
    \cdot\ar[r]^\beta&\cdot\ar[r]^{\beta'}&\cdot
    \enddiagram
  \end{equation}
  The compatibility with the square on the right of~\eqref{eq:91} is the
  requirement that $\bar\beta'\cdot\bar\beta=\widebar{\beta'\cdot\beta}$.
  Together with the square of Example~\ref{ex:6} we obtain all the conditions of
  a split opfibration.
\end{ex}

\section{Existence in the locally presentable case}
\label{sec:exist-locally-pres}

The existence of the cofibrantly \kz-generated \textsc{awfs} over a small
internal
category $\mathbb{J}$ over $\Sq(\C)$ in \VCat\ can be easily deduced
in the case that the underlying category of $\C$ is locally presentable
using techniques from~\cite{MR3393453}. In \S \ref{sec:cofibr-kz-gener} we shall
study cofibrant \kz-generation in another context that encompasses examples
which, as the category of $T_0$ topological spaces, are not locally presentable.

\begin{prop}
  \label{prop:1}
  Suppose that our base of enrichment $\V\subseteq\Cat$ is locally presentable
  ~as an ordinary category~ and it is equipped with a $\V$-enriched
  {\normalfont{\textsl{\textsc{awfs}}}} $(\mathsf{E},\mathsf{M})$ whose
  underlying ordinary {\normalfont{\textsl{\textsc{awfs}}}} is accessible.
  Then, the cofibrantly $(\mathsf{E},\mathsf{M})$-generated {\normalfont\textsl{\textsc{awfs}}} on a small
  internal category $\mathbb{J}$ over $\Sq(\C)$ in \VCat\ exists if $\C$ has
  tensor products with $\two$ and the
  underlying category of $\C$ is locally presentable
\end{prop}
\begin{proof}
  We use \cite[\S 8.2]{MR3393453} to get the result at the level of underlying
  ordinary categories, and then argue why this implies the $\V$-enriched
  version.
  By the mere definitions in \S \ref{sec:cofibrant-generation},
  the underlying category of the $\V$-category
  $\mathbb{J}^{\sperp_{\mathsf{M}}}$ is the category
  $(\mathbb{J}_\uc)^{\sperp_{\mathsf{M}}}$ --~that appears implicitly in \cite[\S
  8.2]{MR3393453}. Therefore
  $(U^{\sperp_{\mathsf{M}}})_\uc\colon
  (\mathbb{J}^{\sperp_{M}})_\uc\to\C^\two_\uc$ has a left adjoint.
  The existence of a \V-enriched left adjoint would be guaranteed if we knew that
  $U^{\sperp_{\mathsf{M}}}$ creates cotensor products.
  The $\V$-functor $U^{\pitchfork_{\mathsf{M}}}\colon
  \mathcal{J}^{\pitchfork_{\mathsf{M}}}\to \C^\two$ automatically creates
  cotensors with $\two$. This because it is the pullback of a monadic
  $\V$-functor along a \V-functor $\hat U$ that preserves any limit that exists
  in $\C^\two$ --~see \S \ref{sec:cofibrant-generation}.
  Recall that $U^{\sperp_{\mathsf{M}}}$ is the composition of the
  equaliser
  $\mathbb{J}^{\sperp_{\mathsf{M}}}\to\mathcal{J}^{\pitchfork_{\mathsf{M}}}$
  with
  $U^{\pitchfork_{\mathsf{M}}}\colon\mathcal{J}^{\pitchfork_{\mathsf{M}}}\to\C^\two$. Therefore,
  it will suffice to prove that $\mathbb{J}^{\sperp_{\mathsf{M}}}$ is closed
  under cotensors with $\two$ in
  $\mathcal{J}^{\pitchfork_{\mathsf{M}}}$. The pair of $\V$-functors
  $\mathcal{J}^{\pitchfork_{\mathsf{M}}}\rightrightarrows
  (\mathcal{J}\times_{\mathcal{J}_0}\mathcal{J})^{\pitchfork_{\mathsf{M}}}$
  whose coequaliser is $\mathbb{J}^{\sperp_{\mathsf{M}}}$ commute with the
  respective $\V$-functors into $\C^\two$, which create cotensors with
  $\two$. It follows that $\mathbb{J}^{\sperp_{\mathsf{M}}}$ is closed
  under these, concluding the proof.
\end{proof}

When $\mathsf{M}$ is the free \textsc{rali} \V-monad, we obtain:
\begin{cor}
  \label{cor:6}
  {\normalfont\textsl{\textsc{awfs}}}s on $\C$ cofibrantly \slkz-generated by
  small internal categories in \VCat\ exist provided that $\C$ has cotensor
  products with $\two$ and its underlying category is locally presentable.
\end{cor}

In \S \ref{sec:lax-orthogonality-kz} we shall show that
cofibrantly $(\mathsf{E},\mathsf{M})$-generated \textsc{awfs}s are \textsc{lofs}s
whenever $(\mathsf{E},\mathsf{M})$ is a \textsc{lofs}. This applies to the
corollary above.
We next give an example in which $\V=\Cat$ and $\mathsf{M}$-algebras are
cloven opfibrations --~see Example~\ref{ex:7}.

\begin{ex}
  \label{ex:10}
  \input{example.tex}
\end{ex}

\section{Lax orthogonality of KZ-cofibrantly generated {AWFS}s}
\label{sec:lax-orthogonality-kz}

When an \textsc{awfs} $(\mathsf{L},\mathsf{R})$ is cofibrantly \kz-generated by
$U\colon \mathcal{J}\to\C^\two$, each $\mathsf{R}$-algebra is \kz-orthogonal to
each $Uj$. What is by no means obvious is that each $\mathsf{R}$-algebra is
\kz-orthogonal to all the $\mathsf{L}$-coalgebras. Ie,
the \textsc{awfs} is lax orthogonal, or a \textsc{lofs}. The proof of this
fact is the subject of the present section.

\begin{theorem}
  \label{thm:1}
  Assume that $\V$ is equipped with a {\normalfont\textsl{\textsc{lofs}}}
  $(\mathsf{E},\mathsf{M})$ and $\mathcal{C}$ is a complete and cocomplete \V-category.
  Any
  {\normalfont\textsl{\textsc{awfs}}} on $\mathcal{C}$ that is cofibrantly $(\mathsf{E},\mathsf{M})$-generated by a small \V-category is lax
  orthogonal. In particular, any {\normalfont\textsl{\textsc{awfs}}} that is
  cofibrantly $\slkz$-generated by a small $\V$-category is lax orthogonal.
\end{theorem}
\begin{proof}
  Suppose that the \V-enriched \textsc{awfs} $(\mathsf{L},\mathsf{R})$ on \C\ is
  cofibrantly $(\mathsf{E},\mathsf{M})$-generated by a \V-category $\mathcal{J}$ over
  $\C^\two$. Then we have that the forgetful \V-functor
  $\mathsf{R}\text-\mathrm{Alg}\to\C^\two$ is the pullback of
  $\mathsf{M}\text-\mathrm{Alg}\to[\mathcal{J}^{\mathrm{op}},\V]^\two$ along
  $W_U\colon\mathcal{C}^\two\to[\mathcal{J}^{\mathrm{op}},\V]^\two$. Corollary~\ref{cor:9} will ensure that $\mathsf{R}$ is lax
  idempotent once we have shown that the right Kan extensions $\Ran_{W_U}$
  exist. The latter condition is obvious since $W_U$ has a left adjoint --~see
  Lemma~\ref{l:14}.
\end{proof}
  The particular case of \kz-generation is obtained by setting
  $(\mathsf{E},\mathsf{M})$ the \textsc{lofs} whose $\mathsf{M}$-algebras are
  \textsc{rali}s.

\begin{theorem}
  \label{thm:2}
  Assume that $\V$ is equipped with a {\normalfont\textsl{\textsc{awfs}}}
  $(\mathsf{E},\mathsf{M})$.
  Suppose that \C\ is a complete \V-category with the property that any small \V-category
  over $\C^\two$ cofibrantly $(\mathsf{E},\mathsf{M})$-generates a {\normalfont \textsl{\textsc{lofs}}}. Then, any
  {\normalfont\textsl{\textsc{awfs}}} cofibrantly $(\mathsf{E},\mathsf{M})$-generated by a small internal category in \VCat\
  is lax orthogonal. In particular, this applies to cofibrant \slkz-generation.
\end{theorem}
\begin{proof}
  Suppose that $\mathbb{J}$ cofibrantly \kz-generates the \textsc{awfs}
  $(\mathsf{L},\mathsf{R})$.
  By considering the equaliser~\eqref{eq:5}, we can exhibit
  $\mathbb{J}^{\sperp_{\mathsf{M}}}$ as the equaliser of a pair of \V-functors,
  as displayed in the left below, where $\mathsf{T}$ and $\mathsf{S}$ are the
  monad part of the
  lax orthogonal \textsc{awfs}s generated by
  $\mathcal{J}$ and
  $\mathcal{J}\times_{\mathcal{J}_0}\mathcal{J}$,
  respectively. In fact, in what follows we only need to regard $\mathsf{R}$,
  $\mathsf{T}$ and $\mathsf{S}$ as ordinary monads on the ordinary category
  $\C^\two_\uc$.
  Since any functor between categories of algebras that
  commutes with the respective forgetful functors is induced by a unique
  morphism of monads,
  there is a commutative diagram in the category
  $\mathbf{Mnd}(\C^\two_\uc)$ as displayed on the right hand side.
  \begin{equation}
    \label{eq:8}
    \diagram
    \mathsf{R}\text-\mathrm{Alg}\ar[r]
    \cong
    \mathbb{J}^{\sperp_{\mathsf{M}}}
    &
    \mathsf{T}\text-\mathrm{Alg}\ar@<4pt>[r]\ar@<-4pt>[r]
    &
    \mathsf{S}\text-\mathrm{Alg}
    \enddiagram
    \qquad
    \diagram
    \mathsf{S}_\uc\ar@<4pt>[r]^-\tau\ar@<-4pt>[r]_-\nu&
    \mathsf{T}_\uc\ar[r]^\zeta&
    \mathsf{R}_\uc
    \enddiagram
  \end{equation}
  By definition, this diagram induces the equaliser diagram at the level of
  corresponding categories of algebras over $\C^\two_\uc$, depicted on the left,
  so it
  exhibits the ordinary monad $\mathsf{R}_\uc$ as the algebraic
  coequaliser of $\tau$ and $\nu$, and by~\cite[Prop.~26.2]{MR589937}, as the
  coequaliser of this pair of morphisms in the category $\mathbf{Mnd}(\C^\two_\uc)$
  of monads on $\C^\two_\uc$.

  Lax idempotent 2-monads on a 2-category $\mathcal{A}$ are characterised by
  being co-orthogonal to certain 2-monad morphisms $\sigma_f$, for each morphism
  $f$ in $\mathcal{A}$, as shown in~\cite{Kelly:Prop-like} and recalled in
  \S \ref{sec:refl-lax-idemp}.
  Thus, we have to prove that $\mathsf{R}_\uc$ is co-orthogonal to these morphisms
  $\sigma_{(h,k)}$, for $(h,k)$ a morphism in $\C^\two$. As
  the full subcategory of objects co-orthogonal
  to each member of a family of morphisms is always closed under colimits,
  $\mathsf{R}_\uc$ is co-orthogonal to each $\sigma_{(h,k)}$, completing the
  proof.
\end{proof}

Under certain conditions, the hypotheses of Theorem~\ref{thm:2} are always
satisfied, as for example:
\begin{cor}
  \label{cor:2}
  Suppose that $(\mathsf{E},\mathsf{M})$ is an accessible
  \V-enriched {\normalfont\textsl{\textsc{lofs}}} on the locally presentable $\V$, and that
  the underlying category of the complete and cocomplete \V-category \C\ is
  locally presentable. Then, any small internal category $\mathbb{J}$ in \VCat\
  over $\Sq(\C)$ cofibrantly $(\mathsf{E},\mathsf{M})$-generates a \V-enriched
  {\normalfont\textsl{\textsc{awfs}}} which, moreover, is lax orthogonal.  In
  particular, this applies to \slkz-generation.
\end{cor}
The proof of the corollary consists of a simple application of
Proposition~\ref{prop:1} and Theorem~\ref{thm:2}.

\begin{cor}
  \label{prop:8}
  The {\normalfont\textsl{\textsc{awfs}}}s whose right morphisms are,
  respectively, opfibrations, normal opfibrations, and split opfibrations are
  lax orthogonal.
\end{cor}
\begin{proof}
  The three \textsc{awfs}s are of the form $\mathcal{J}^{\pitchfork_{\mathkz}}$ for a
  small category $\mathcal{J}$ or
  $\mathbb{J}^{\sperp_{\mathkz}}$ for a small double category $\mathbb{J}$, as
  seen in Examples~\ref{ex:7}, \ref{ex:6} and~\ref{ex:8}.
\end{proof}

\section{Representable multicategories}
\label{sec:repr-cat-enrich}

\input{multicat2.tex}

\section{Cofibrant generation and accessibility}
\label{sec:cons-cofibr-gener}

If $\mathbb{J}$ is a small double category over $\Sq(\C)$ and $\C$ is locally
presentable, \cite{MR3393453} showed that
$\mathbb{J}^{\sperp}\to\mathcal{C}^\two$ creates $\kappa$-filtered colimits, for
some $\kappa$. In this section, we look at the case when $\C$ is not necessarily
locally presentable but, following \cite{MR589937}, it is equipped with an
\textsc{ofs}. Our main example will be the category of topological spaces.
The section's results will be used later on to give examples of \textsc{lofs}s
that are not cofibrantly (\kz-)generated.

  Let $\mathcal{M}$ be a subcategory of a category $\mathcal{A}$. By an
  \emph{$\mathcal{M}$-diagram} in
  $\mathcal{A}$ we will mean a functor $\mathcal{D}\to\mathcal{A}$ that factors
  through $\mathcal{M}$ and has small domain. This $\mathcal{M}$-diagram will be
  said to be $\kappa$-filtered, for a regular cardinal $\kappa$, if
  $\mathcal{D}$ is a $\kappa$-filtered category.

  \begin{df}
    Let $\mathcal{A}$ be a cocomplete category with a subcategory
    $\mathcal{M}$. A functor with domain $\mathcal{A}$ is
    \emph{$\mathcal{M}$-accessible} if it preserves $\kappa$-filtered colimits
    of $\mathcal{M}$-diagrams, for some regular cardinal $\kappa$.
  \end{df}
An object $A$ of
$\mathcal{A}$ has $\mathcal{M}$-\emph{rank} less or equal to a regular cardinal
$\kappa$ if the representable functor $\mathcal{A}(A,-)$ is $\mathcal{M}$-accessible.
The $\mathcal{M}$-rank of $A$ is the smallest
regular cardinal $\kappa$ for which this happens.
When there is no risk of confusion, we may omit $\mathcal{M}$ and simply say
rank.
The notion of rank of an object has been attributed by~\cite{MR0322004} to M.~Barr.
We say that $\mathcal{A}$ is \emph{locally
    $\mathcal{M}$-ranked} if each of its objects has a
  rank.
The subcategory $\mathcal{M}$ will most often
be the right part of an \textsc{ofs}
$(\mathcal{E},\mathcal{M})$.
Categories locally bounded with respect to an
\textsc{ofs} $(\mathcal{E},\mathcal{M})$ --~in the sense  of~\cite[\S 6.1]{Kelly:BCECT}~-- are locally
ranked for $\mathcal{M}$, due to the argument given in~\cite[\S
3.2]{MR0322004}. A number of
examples are given in~\cite{MR0322004} and \cite[\S 6.1]{Kelly:BCECT}; later we
will be interested in the example of the category of $T_0$ topological spaces
equipped with the \textsc{ofs} $\mathcal{E}$ = surjections and
$\mathcal{M}$ = subspace inclusions.

Each subcategory $\mathcal{M}$ of $\A$ induces
a subcategory of $\A^\two$ component-wise, ie the subcategory consisting of
those morphisms in $\A^\two$ with both domain and codomain components in
$\mathcal{M}$. We will still denote this subcategory by $\mathcal{M}$ when no
confusion is likely.
It is easy to show that a morphism $f\colon X\to Y$ has rank less or equal to $\kappa$ as an object of
$\A^\two$ if $X$ and $Y$ have rank less or equal to $\kappa$.

For an object $X$ of a \V-category \C\ and a subcategory
$\mathcal{M}\subseteq\C_\uc$, one has two related notions. First, one can say
that $\C(X,-)\colon\C\to\V$ preserves \V-enriched colimits of $\kappa$-filtered
$\mathcal{M}$-diagrams. This is the same as asserting that the ordinary functor
$\C(X,-)\colon\C_\uc\to\V$ preserves these colimits.
Secondly, one can say that $X$ has $\mathcal{M}$-rank
less or equal to $\kappa$, or what is the same, that
$\C_\uc(X,-)\colon\C_\uc\to\mathbf{Set}$ preserves $\kappa$-filtered colimits of
$\mathcal{M}$-diagrams. The former implies the latter but not the other way
around, not in general. Before stating a standard lemma looking at the
relationship between these conditions, we recall a piece of
terminology: an object $X\in\V$ is finitely presentable when
$\V(X,-)\colon\V\to\mathbf{Set}$ preserves filtered colimits, ie when it has
rank $\aleph_0$.

\begin{lemma}
  \label{l:18}
  Assume, in addition to our blanket hypotheses on \V, that the inclusion $\V\hookrightarrow\Cat$ is finitary.
  Let $\C$ be a cocomplete $\V$-category with a subcategory
  $\mathcal{M}\subseteq\C_\uc$. Denote by $(\C_\uc)_\kappa\subset\C_\uc$ the
  full subcategory of objects of $\mathcal{M}$-rank less or equal to $\kappa$.
  For $X\in\C$, the functor $\C(X,-)\colon\C_\uc\to\V$ preserves
  $\kappa$-filtered colimits of $\mathcal{M}$-diagrams if and only if $\two*X\in(\C_\uc)_\kappa$.
\end{lemma}
\begin{proof}
  The first observation is that $J\colon \V\hookrightarrow\Cat$ is
  finitary if and only if $\two$ is a finitely presentable object of
  \V. Indeed, the arrow category $\two$ is a strong generator in $\Cat$, so
  $\Cat(\two,-)\colon\Cat\to\mathbf{Set}$ preserves and reflects filtered
  colimits. Then, $J$ is finitary if and only if its composition with
  $\Cat(\two,-)$, which is isomorphic to
  $\V(\two,-)\colon\V\to\mathbf{Set}$, is so.

  The functor $\C(\two*X,-)\colon\C_\uc\to\mathbf{Set}$ is isomorphic to the
  composition of the representables $\C(X,-)\colon \C_\uc\to\V$ and
  $\V(\two,-)\colon\V\to\mathbf{Set}$. The latter preserves and reflects
  filtered colimits. It follows that $\C(\two*X,-)$ and $\C(X,-)$ preserve the
  same filtered colimits.
\end{proof}

\begin{Ass}
  \label{hypothesis:1}
  In addition to our blanket hypotheses on \V\ (it is closed under limits and
  exponentials in \Cat, it is cocomplete and $\two\in \V$) we shall, for the
  rest of this section,
  assume that it is closed under filtered colimits in \Cat. As seen in the
  previous lemma's proof, this new condition is equivalent to requiring that
  $\two$ should be a finitely presentable object of \V. Furthermore, it makes
  $\V$ a finitely locally presentable category. Our main examples will be
  $\V=\Cat$ and $\V=\Ord$ the category of posets.
\end{Ass}

\begin{lemma}
  \label{l:19}
  Given a morphism $j\colon X\to Y$ in \C, the functor $\C^\two(j,-)\colon\C^\two_\uc\to\V$
  preserves $\kappa$-filtered colimits of $\mathcal{M}$-diagrams if and only if
  $\two*X$ and $\two*Y$ lie in $(\C_\uc)_\kappa$.
\end{lemma}
\begin{proof}
  The functor $\C^\two(j,-)$ can be constructed as the pullback of the diagram
  \begin{equation}
    \diagram
    \C(X,\dom-)\ar[r]&\C(X,\cod-)&\C(Y,\cod-)\ar[l]
    \enddiagram
  \end{equation}
  so it will preserve any filtered colimit that is preserved by $\C(X,-)$ and
  $\C(Y,-)$. The result now
  follows from Lemma~\ref{l:18}.
\end{proof}

\begin{df}
  A $\V$-enriched \textsc{awfs} $(\mathsf{E},\mathsf{M})$ on a cocomplete
  $\V$-category is \emph{$\mathcal{M}$-accessible} if one of the following
  equivalent conditions holds: the comonad $\mathsf{E}$ is
  $\mathcal{M}$-accessible; the monad $\mathsf{M}$ is $\mathcal{M}$-accessible;
  the \V-functor $K=\dom M=\cod E \colon\C^\two\to\C$ is
  $\mathcal{M}$-accessible. When $\mathcal{M}$ is the whole category, one says
  that the \textsc{awfs} is accessible.
\end{df}.

\begin{prop}
  \label{prop:2}
  Suppose \V\ is equipped with an accessible {\normalfont\textsl{\textsc{awfs}}}
  $(\mathsf{E},\mathsf{M})$ and that it satisfies the hypotheses in
  Assumption~\ref{hypothesis:1}. Suppose that
  $\C$ is a cocomplete \V-category whose underlying category $\C_\uc$ is locally
  ranked with respect to the subcategory $\mathcal{M}$.
  If $\mathbb{J}$ is a small internal category in \VCat\
    over $\Sq(\C)$, then $\mathbb{J}^{\sperp_{\mathsf{M}}}\to\C^\two$ creates
  $\kappa$-filtered colimits of $\mathcal{M}$-diagrams, for some regular
  cardinal $\kappa$.
\end{prop}
\begin{proof}
  We first prove the case of cofibrant generation by a small
  category $\mathcal{J}$ over $\C^\two$.
  By definition --~see diagram~\eqref{eq:131}~--
  $\mathcal{J}^{\pitchfork_{\mathsf{M}}}\to\C^\two$ is the pullback of the
  forgetful $V\colon
  \mathsf{M}\text-\mathrm{Alg}\to[\mathcal{J}^{\mathrm{op}},\V]^\two$ along
  $W_U$.
  The  enriched monad $\mathsf{M}$ is cocontinuous, so $V$
  creates all colimits and, by Lemma~\ref{l:20}, it only remains to verify that $W_U$ preserves
  $\kappa$-filtered colimits of $\mathcal{M}$-diagrams, for some $\kappa$.

  Since $\mathcal{J}$ is small, there is a regular cardinal $\kappa$ such that
  each $Uj\in\C^\two$ has rank less or equal to $\kappa$, ie $\C^\two(Uj,-)$ preserves
  $\kappa$-filtered colimits of $\mathcal{M}$-diagrams. This means that $\tilde
  U\colon\C^\two\to[\mathcal{J}^{\mathrm{op}},\V]$, given by $f\mapsto\C^\two(U-,f)$, preserves
  $\kappa$-filtered colimits of $\mathcal{M}$-diagrams. Furthermore, the functor
  $W_U$ sends $f$ to the (dashed) comparison morphism depicted
  in~\eqref{eq:13}, from where it is clear that $W_U$ preserves
  $\kappa$-filtered colimits of $\mathcal{M}$-diagrams.

  Now we prove the case of cofibrant generation by an internal category
  $\mathbb{J}=(\mathcal{J}\rightrightarrows\mathcal{J}_0)$
  in \VCat\ and $U=(U_1,U_0)\colon\mathbb{J}\to\Sq(\C)$ an internal
  functor. Denote by $V$ the functor
  $\mathcal{J}\times_{\mathcal{J}_0}\mathcal{J}\to\C^\two$ from the object of
  composable pairs. By definition, $\mathbb{J}^{\sperp_{\mathsf{M}}}$ is an
  equaliser \eqref{eq:5} in $\VCat/\C^\two$ of a pair of \V-functors of the form
  $\mathcal{I}_1^{\pitchfork_{\mathsf{M}}}\rightrightarrows
  \mathcal{I}_2^{\pitchfork_\mathsf{M}}$. Each of the two respective \V-functors
  into $\C^\two$ create $\kappa$-filtered colimits of $\mathcal{M}$-diagrams,
  for some $\kappa$. We choose the largest of these two cardinals, so both \V-functors
  create $\kappa$-filtered colimits of $\mathcal{M}$-diagrams. It follows that
  the parallel \V-functors strictly preserve the created colimits and that
  $\mathbb{J}^{\pitchfork_{\mathsf{M}}}$ is closed under these in
  $\mathcal{I}_1^{\pitchfork_{\mathsf{M}}}$, completing the proof.
\end{proof}
\begin{cor}
  \label{cor:3}
  Suppose that $\V$ and $\C$ are as in Proposition~\ref{prop:2}. Then, any
  cofibrantly $(\mathsf{E},\mathsf{M})$-generated
  {\normalfont\textsl{\textsc{awfs}}} on \C\ is $\mathcal{M}$-accessible.
\end{cor}

The reader would remember that the notation $\mathcal{J}^\pitchfork$, without mention of the
\V-monad $\mathsf{M}$, means that we are taking as $\mathsf{M}$ the \V-monad whose
algebras are split epimorphisms. The objects of $\mathcal{J}^\pitchfork$ were
described in Example~\ref{rmk:3}.
\begin{lemma}
  \label{l:6}
  Let $(\mathsf{L},\mathsf{R})$ be an {\normalfont\textsl{\textsc{awfs}}} on the \V-category \C, with
  underlying {\normalfont\textsl{\textsc{wfs}}} $(\mathcal{L},\mathcal{R})$, and
  $\mathcal{I}\subset\mathcal{L}$ a small set of morphisms. There exists a
  \V-functor over $\C^\two$
  \begin{equation}
    \mathsf{R}\text-\mathrm{Alg}\longrightarrow\mathcal{I}^\pitchfork
    .
  \end{equation}
\end{lemma}
\begin{proof}
  The class $\mathcal{L}$ consists of those morphisms of $\C$ that admit at
  least one coalgebra structure for the copointed endo-\V-functor
  $(L,\Phi)$. Since $\mathcal{I}$ is small, we can choose a coalgebra structure for each
  $i\in\mathcal{I}$, or equivalently, a \V-functor over
  $\C^\two$
  \begin{equation}
    \label{eq:130}
    \mathcal{I}\longrightarrow(\mathsf{L},\Phi)\text-\mathrm{Coalg}
  \end{equation}
  with domain the discrete \V-category associated to the set $\mathcal{I}$.
  We can now consider the composition of functors
  \begin{equation}
    \label{eq:129}
    \mathsf{R}\text-\mathrm{Alg} \hookrightarrow{}
    (R,\Lambda)\text-\mathrm{Alg} \longrightarrow
    (L,\Phi)\text-\mathrm{Coalg}^{\pitchfork} \longrightarrow
    \mathcal{I}^\pitchfork
  \end{equation}
  where the first one is an inclusion, the second is the \V-functor
  of~\eqref{eq:93} --~in fact, \eqref{eq:93}~describes an ordinary functor that
  can easily be reproduced in the \V-enriched context~-- and the last one is the
  \V-functor induced by~\eqref{eq:130}.
\end{proof}
  The construction of the above lemma
  is valid only because $\mathcal{I}$ is a set instead of a category,
  in which case we would not be able to guarantee the compatibility between the
  coalgebra structures chosen for each $i\in\mathcal{I}$ and the morphisms of
  $\mathcal{I}$.

Given an \textsc{awfs}
$(\mathsf{L},\mathsf{R})$ on a \V-category \C, we denote by $\mathsf{R}_1$ its \emph{fibrant replacement monad}, ie the
restriction of $\mathsf{R}$ to $\C/1\cong\C$.

We close the section with a proposition that shall be used in \S
\ref{sec:continuous-lattices-1} to exhibit a new example of a \textsc{lofs} that
is not cofibrantly generated.

\begin{prop}
  \label{prop:4}
  Suppose that \V\ satisfies Assumption~\ref{hypothesis:1} and that $\C$ is a \V-category satisfying:
  \begin{enumerate*}[label=(\alph*)]
  \item Its underlying category $\C_\uc$ is cocomplete and is locally ranked
    with respect to the subcategory $\mathcal{M}$.
  \item It has a terminal object.
  \item It is equipped with a \V-enriched {\normalfont\textsl{\textsc{awfs}}}
    $(\mathsf{L},\mathsf{R})$ whose underlying
    {\normalfont\textsl{\textsc{wfs}}} $(\mathcal{L},\mathcal{R})$ is
    cofibrantly generated by a set of morphisms with $\mathcal{M}$-rank.
  \end{enumerate*}
  Then, there is a regular cardinal $\kappa$ such
  that any $\kappa$-filtered colimit $\operatorname{colim}X_j$ of an $\mathcal{M}$-diagram
  is a fibrant object for $(\mathcal{L},\mathcal{R})$, provided that ${X_j}$ is
  a diagram in $\mathsf{R}_1\text-\mathrm{Alg}$.
\end{prop}
\begin{proof}
  Suppose that $\mathcal{I}\subset\mathcal{L}$ cofibrantly generates
  $(\mathcal{L},\mathcal{R})$. 
  Consider a \V-functor $\mathsf{R}\text-\mathrm{Alg}\to\mathcal{I}^\pitchfork$
  over $\C^\two$ as provided by Lemma~\ref{l:6}. Taking the fibre over the terminal
  object $1\in\C$, we obtain the first arrow displayed below.
  \begin{equation}
    \label{eq:32}
    \mathsf{R}_1\text-\mathrm{Alg}\longrightarrow
    (\mathcal{I}^\pitchfork)_1\longrightarrow \mathbf{Fib}
  \end{equation}
  The second arrow is the inclusion into the full subcategory of fibrant
  objects, which exists since any object with a lifting operation against
  $\mathcal{I}$ is certainly weakly orthogonal to $\mathcal{I}$.
  Proposition~\ref{prop:2} guarantees that the forgetful
  $\mathcal{I}^\pitchfork\to\C^\two$ creates $\kappa$-filtered colimits of
  $\mathcal{M}$-diagrams for some regular cardinal $\kappa$. If
  $\{X_j\}$ is a $\kappa$-filtered diagram in $\mathsf{R}_1\text-\mathrm{Alg}$ whose
  morphisms also lie in $\mathcal{M}$, and if $\operatorname{colim}X_j$ exists
  in \C, we can send
  $\{X_j\}$ along~\eqref{eq:32} to a diagram in $(\mathcal{I}^\pitchfork)_1$,
  and deduce that $\colim X_j$ is created by
  $(\mathcal{I}^\pitchfork)_1\to\C$. In particular, $\colim X_j$ supports a structure of an object of
  $(\mathcal{I}^\pitchfork)_1$, thus it is a fibrant object for
  $(\mathcal{L},\mathcal{R})$.
\end{proof}

\section{Order-enriched \textsc{awfs}s}
\label{sec:cofibr-kz-gener}

We now turn our attention to the case of order-enriched \textsc{awfs}s and
their cofibrant \kz-generation. The case of a locally presentable base category
was dealt with in \S \ref{sec:exist-locally-pres}, so we now
concentrate in the non-locally presentable case, as to include important examples
as that of the category of topological spaces. The section ends with a short
comparison with \cite{MR3283679}.

We will denote the category of
posets by \Pord. By poset we mean a set equipped with a partial ordering that is
antisymmetric, or equivalently, a small category that has at most one morphism
between any two objects and whose isomorphisms are identity morphisms.
\begin{Ass}
  \label{ass:1}
  Let us assume that our category $\V\subseteq\Cat$ is
  the cartesian closed category $\Pord$ of posets. We equip it with an accessible
  \textsc{lofs} $(\mathsf{E},\mathsf{M})$.
\end{Ass}

\begin{lemma}
  \label{l:1}
  Let $\mathsf{T}=(T,\eta,\mu)$ be a lax idempotent monad on an \Ord-category
  $\mathcal{A}$. The forgetful \Ord-functor
  $\mathsf{T}\text-\mathrm{Alg}\to (T,\eta)\text-\mathrm{Alg}$ is
  an isomorphism.
\end{lemma}
This lemma is implicit in, for example, \cite[Cor.~4.2.3]{MR1641443}. To give a
short proof, recall that the identity $1_{TA}$ is a left extension of $\eta_A$
along $\eta_A$, since $T$ is lax idempotent. If $a\colon TA\to A$ satisfies
$a\cdot \eta_A=1_A$, it will be a $T$-algebra structure as soon as $a\dashv
\eta_A$. All that is left to show is that $1_{TA}\leq \eta_A\cdot a$, which is
equivalent to $\eta_A\leq \eta_A\cdot a\cdot \eta_A$, which holds true.

\begin{lemma}
  \label{l:2}
  If $\mathbb{J}=(\mathcal{J}\rightrightarrows\mathcal{J}_0)$ is an internal
  category in $\Pord\text-\Cat$ over $\Sq(\C)$, then the inclusion
  $\mathbb{J}^{\sperp_{\mathsf{M}}}\hookrightarrow\mathcal{J}^{\pitchfork_{\mathsf{M}}}$ is an identity.
\end{lemma}
\begin{proof}
  By definition, the inclusion is fully faithful and injective on objects. It
  remains to show that is surjective on objects. $\mathsf{M}$-algebra
  structures are unique --~see \S \ref{sec:lax-idempotent-2}. Inspecting the definition of fibre square in \S
  \ref{sec:cofibrant-generation}, we have that fibre square structures on
  objects of $[\mathcal{J}^{\mathrm{op}},\Ord]^{\two\times\two}$ are unique. By
  definition, the objects of $\mathbb{J}^{\sperp_{\mathsf{M}}}$ are the objects
  of $\mathcal{J}^{\pitchfork_{\mathsf{M}}}$ such that a certain pair of fibre
  square structures on the square~\eqref{eq:4} coincide. This holds, by
  uniqueness, so
  $\mathbb{J}^{\sperp_{\mathsf{M}}}=\mathcal{J}^{\pitchfork_{\mathsf{M}}}$.
\end{proof}

An \textsc{ofs} $(\mathcal{E},\mathcal{M})$ on a cocomplete category
$\mathcal{A}$ is \emph{cocomplete} if wide pushouts of morphisms in
$\mathcal{E}$ exist \cite{MR589937}. In other words, for each object $A\in\mathcal{A}$, the
category $A\downarrow\mathcal{E}$ has arbitrary coproducts. We are now ready to
prove the section's main result.

\begin{theorem}
  \label{thm:3}
  Let $\C$ be a cocomplete and finitely cocomplete \Pord-enriched category whose underlying
  category satisfies:
  \begin{enumerate*}[label=(\alph*)]
  \item It has a cocomplete {\normalfont\textsl{\textsc{ofs}}}
    $(\mathcal{E},\mathcal{M})$.
  \item It is locally $\mathcal{M}$-ranked.
  \end{enumerate*}
  Then, the
  cofibrantly $(\mathsf{E},\mathsf{M})$-generated
  {\normalfont\textsl{\textsc{awfs}}} on any internal category in $\Pord\text-\Cat$ over $\Sq(\C)$ exists and
  is lax orthogonal.
\end{theorem}
\begin{proof}
  It suffices to consider the case of a small \Pord-category
  $U\colon\mathcal{J}\to\C^\two$, by Lemma~\ref{l:2}.
  By definition, $\mathcal{J}^{\pitchfork_{\mathsf{M}}}\to\C^\two$ is the pullback
  of $\mathsf{M}\text-\mathrm{Alg}\to[\mathcal{J}^{\mathrm{op}},\Pord]^{\two}$
  along $W_U\colon\mathcal{C}^\two\to [\mathcal{J}^{\mathrm{op}},\Pord]^{\two}$,
  or
  equivalently, by Lemma~\ref{l:1}, the pullback of
  $(M,\Lambda^M)\text-\mathrm{Alg}\to[\mathcal{J}^{\mathrm{op}},\Pord]^\two$. Since
  $W_U$ has a left adjoint $G\dashv W_U$, the \Pord-category
  $\mathcal{J}^{\pitchfork_{\mathsf{M}}}$ is isomorphic to $(R,\Lambda^R)\text-\mathrm{Alg}$
  for the pointed \Pord-functor $(R,\Lambda^R)$ on $\C^\two$ given by the pushout
  depicted on the right.
  \begin{equation}
    \label{eq:10}
    \diagram
    \mathcal{J}^{\pitchfork_{\mathsf{M}}}\ar[r]\ar[d]&
    (M,\Lambda^M)\text-\mathrm{Alg}\ar[d]\\
    \C^\two\ar[r]^-{W_U}&
    [\mathcal{J}^{\mathrm{op}},\Pord]^\two
    \enddiagram
    \qquad
    \diagram
    GW_U\ar[d]_{G\Lambda^M W_U}\ar[r]^-{\mathrm{counit}}&1\ar[d]^{\Lambda^R}\\
    GMW_U\ar[r]&R
    \enddiagram
  \end{equation}

  To conclude the proof, it suffices to show that the functor
  \begin{equation}
    (R,\Lambda^R)\text-\mathrm{Alg}
    \longrightarrow
    \C^\two
    \label{eq:27}
  \end{equation}
  is monadic. In fact, it suffices to show that its underlying ordinary functor
  has a left adjoint, by Lemma~\ref{l:17}.
  To do so, we can show that $W_U$ is $\mathcal{M}$-accessible, in the
  same way as we did in the proof of Proposition~\ref{prop:2}.
  We then observe that $GW_U$ and $GMW_U$ are $\mathcal{M}$-accessible too, $G$
  is cocontinuous and $M$ is accessible. Therefore, $R$ is
  $\mathcal{M}$-accessible. We may now use~\cite[14.3 \& 15.6]{MR589937} to
  deduce that
  the ordinary functor \eqref{eq:27} has a left adjoint.
\end{proof}
\begin{rmk}
  \label{rmk:5}
Theorem~\ref{thm:3} applies to the case of the \textsc{lofs}
$(\mathsf{E},\mathsf{M})$ for which $\mathsf{M}$ is the free \textsc{rali} monad
on $\V^\two$ of \S \ref{sec:intern-categ-laris}. In this case it guarantees that
cofibrantly \kz-generated \textsc{lofs}s exist on any \Ord-category
whose underlying ordinary category is locally ranked with respect to a
cocomplete \textsc{ofs}.
\end{rmk}
\begin{rmk}
  \label{rmk:4}
  The version of Theorem~\ref{thm:3} mentioned in the previous remark is related
  to \cite{MR3283679}. One of its main results, \cite[Thm~6.10]{MR3283679},
  states that, if $\C$ is locally ranked and $\mathcal{J}$ is a family of
  morphisms of \C, the full subcategory of $\mathcal{C}$ defined by the objects
  that are \kz-injective to each morphism of $\mathcal{J}$ is
  reflexive. Furthermore, the induced monad is lax-idempotent. The main
  component of the proof is the construction of the Kan-injective reflection
  chain in~\cite[\S 5]{MR3283679}. When $\mathcal{C}$ has a terminal object,
  this can be translated to our notation by saying that the the fibre of the
  codomain functor $\mathcal{J}^{\pitchfork_{\mathkz}}\to\C^\two\to\C$ over
  $1\in\C$ is monadic over $\C$, with lax-idempotent induced monad.

  A point where \cite{MR3283679} is more general than the present paper is that
  the class of morphism $\mathcal{J}$ need not be small, but only those which
  are not order-epimorphisms should form a small set.

  On the other hand, our approach is more general in a couple of ways. The only
  part of this section that is specific to $\V=\Ord$ is the existence part of
  Theorem~\ref{thm:3}. The lax orthogonality of the resulting \textsc{awfs} is
  valid in more general cases, for example $\V=\Cat$ (see \S
  \ref{sec:lax-orthogonality-kz}). Another advantage of our approach is that it
  accommodates an $\Ord$-category $\mathcal{J}$ over $\C^\two$, instead of just
  a family of morphisms of \C. Finally, even though our requirements on $\C$ and
  the \textsc{ofs} $(\mathcal{E},\mathcal{M})$ differ slightly from those
  of~\cite{MR3283679}, this difference seems to be superfluous in the the
  examples of interest.
\end{rmk}

\section{Continuous lattices}
\label{sec:continuous-lattices-1}
The papers \cite{1702.02602} and \cite{lmcs3960} constructed an \textsc{lofs} on
the category of $T_0$ topological spaces by means of the ``method of the simple
monad'' introduced in~\cite{Clementino2016458}. In this section we shall show
that this \textsc{lofs}, which is closely related to complete lattices, is not
cofibrantly \kz-generated, nor cofibrantly generated, nor is its underlying
\textsc{wfs} cofibrantly generated.
We begin with a brief summary of continuous lattices.

Let \Topo\ be the category of $T_0$ topological spaces and continuous maps. Each
$T_0$ space $X$ carries a posetal structure defined by $x\sqsubseteq y$ if and
only if each neighbourhood of $x$ is also a neighbourhood of $y$; equivalently,
if $x\in\overline{\{y\}}$. This is sometimes called the specialisation order. A
continuous function $f\colon X\to Y$ becomes an order-preserving function
$f\colon(X,\sqsubseteq)\to(Y,\sqsubseteq)$, so we have a functor
\( \Topo\to\Ord\) into the category of posets.
The cartesian closed category $\Ord$ can play the role of the base
of enrichment \V\ of the previous sections, and we can make $\Topo$ into a $\Ord$-category
by declaring $f\leq g$ if the associated morphisms of posets satisfy
$f\leq g$. 
In
elementary terms, $f\leq g$ if and only if $f(x)\sqsubseteq g(x)$ for all
$x$.

We now recall the definition of continuous lattices~\cite{MR0404073}. Suppose
that $(L,\leq)$ is a complete poset. Given a pair of elements $x$, $y\in
L$ we say that $x$ is \emph{way below}
$y$, written $x\ll y$, if for all directed subsets $D\subseteq L$, if $y\leq \vee D$ then
there exists some $d\in D$ with $x\leq d$. A \emph{continuous
lattice} is a complete poset where every element is the supremum of the
elements way below it:
$
  x=\bigvee\{y:y\ll x\}.
$
Equip $L$ with the Scott topology $\tau_{L}$, whose open sets are those subsets $U\subseteq L$
that satisfy:
\begin{enumerate*}
\item \label{item:17} if $x\in U$ and $x\leq y$, then $y\in U$;
\item \label{item:18} if $D\subseteq L$ is a directed subset and $\bigvee D\in
  U$, then $D\cap U\neq\emptyset$.
\end{enumerate*}
In this way we can regard any continuous lattice as a $T_0$ topological
space, and the specialisation order for this topology coincides with the
order of $L$. Conversely, if the poset $(X,\sqsubseteq)$ of a $T_0$ space
$X$ is a continuous lattice, the topology $\tau_{(X,\sqsubseteq)}$ coincides
with the original topology of $X$. In this way, continuous lattices can be
identified with a class of $T_0$ topological spaces.

A function $f\colon L\to L'$ between continuous lattices is continuous for the
associate topology if and only if it preserves directed suprema. Thus, the
category of continuous lattices and maps of posets that preserve directed suprema
is isomorphic to a full subcategory $\mathbf{CL}$ of $\Topo$.

As part of his seminal work~\cite{MR0404073},
D.~Scott showed that a topological space is a continuous lattice if
and only if it is injective with respect to all embeddings of topological
spaces. Later, it was shown in~\cite{MR0367013} that the category of continuous lattices is
isomorphic to the category of algebras of the filter monad $\mathsf{F}$ on
$\Topo$, the monad that assigns to each space $X$ its space of filters $FX$
endowed with a suitable topology.

Before explaining how all this relates to \textsc{lofs}, let us make a point
about the direction of the inequalities between continuous maps. The filter
monad on $\Topo$ is colax idempotent, with the ordering between continuous maps
induced by the specialisation order. When we want to apply the theory of
\textsc{lofs} and simple monads to this example, we are presented with two
options: we either reverse the ordering on 2-cells (inequalities) in our statements, or else
we reverse the specialisation ordering to make the filter monad lax
idempotent. We choose the latter approach, used in~\cite{MR2927175,1702.02602},
so $f\leq g$ for a pair of parallel continuous maps shall mean
$f(x)\sqsupseteq g(x)$ for all $x$ in their domain.

We are now ready to explain how to construct an \textsc{lofs}, as done
in~\cite{1702.02602,lmcs3960}. The filter monad $\mathsf{F}$ is lax
idempotent --~with the convention on the ordering of continuous maps of the
previous paragraph. Furthermore, $\mathsf{F}$ is \emph{simple} in the terminology of
\cite{Clementino2016458}, thus inducing a \textsc{lofs} $(\mathsf{L},\mathsf{R})$ on
\Topo\ such that:
\begin{enumerate*}
\item \label{item:21} the fibrant replacement $\Ord$-enriched monad on
  $\Topo$ is the filter monad $\mathsf{F}$;
\item \label{item:22} a continuous function $f$ is an $\mathsf{L}$-coalgebra if
  and only if $f$ is a topological embedding.
\end{enumerate*}
A detailed construction can be found in~\cite{1702.02602}.
The underlying \textsc{wfs} of this \textsc{lofs} was considered in~\cite{MR2927175}.
\begin{theorem}
  \label{thm:11}
  The {\textsl{\textsc{lofs}}} on \Topo\ described is {not}
  cofibrantly $(\mathsf{E},\mathsf{M})$-generated, for any cocontinuous
  $\Ord$-enriched {\textsl{\textsc{awfs}}s} $(\mathsf{E},\mathsf{M})$ on
  \Ord. In particular it is not \slkz-generated nor cofibrantly
  generated. Furthermore, its underlying
  {\normalfont\textsl{\textsc{wfs}}} is not
  cofibrantly generated.
\end{theorem}
\begin{proof}
  We first tackle the part of the statement that deals with the \textsc{awfs}.
  The proof is an application of Proposition~\ref{prop:2}. The category \C\ will
  be that of $T_0$ topological spaces, $\Topo$, regarded as an $\Ord$-category
  via the specialisation order, and $\mathcal{M}$ will be the the subcategory of
  $\Topo$ defined by continuous maps that are injective and
  homeomorphisms onto their image. It is well-known
  that each topological space has an $\mathcal{M}$-rank.

  Suppose that the \textsc{awfs} $(\mathsf{L},\mathsf{R})$ described at
  the beginning of the section is cofibrantly generated with respect to a
  cocontinuous \textsc{awfs} $(\mathsf{E},\mathsf{M})$ on \Ord.
  By Proposition~\ref{prop:2}, the
  forgetful functor $\mathsf{R}\text-\mathrm{Alg}\to\C^\two$ creates
  $\kappa$-filtered colimits of $\mathcal{M}$-diagrams, for some regular cardinal
  $\kappa$. In particular, if $\mathsf{R}_1$ is the associated fibrant
  replacement monad, ie the restriction of $\mathsf{R}$ to $\Topo/1\cong\Topo$,
  then the forgetful functor $\mathbf{CL}\cong \mathsf{R}_1\text-\mathrm{Alg}\to\Topo$ creates
  the said colimits. This is a contradiction, as exhibited by the following
  example, which, furthermore, together with Proposition~\ref{prop:4} shows that
  the underlying \textsc{wfs} of $(\mathsf{L},\mathsf{R})$ cannot be cofibrantly
  generated.
\end{proof}

\begin{ex}
  \label{ex:5}
  In the next few paragraphs we show that, for each regular cardinal $\beta$,
  there is a $\beta$-filtered colimit of continuous lattices that is \emph{not} created
  by the forgetful functor $\mathbf{CL}\to\Topo$. As usual, $\alpha<\nu$ will
  mean $\alpha\in \nu$ for ordinals $\alpha$, $\nu$.


  The way-below relation $\alpha\ll\nu$ on an ordinal $\beta$ satisfies: 
  if $\alpha\in\nu\in\beta$, then $\alpha\ll\nu$. For, if
  $D\subset\beta$ verifies $\nu\leq \sup D$, then $\alpha\in\sup
    D$ and there must be a $\delta\in D$ with $\alpha\in\delta$.

  It is now easy to verify that successor ordinals are continuous lattices, as
  they are complete and any $\alpha$ is the supremum of the ordinals
  $\gamma\in\alpha$. If $\mu\in\nu$, the inclusion
  $\mathrm{succ}(\mu)\subset\mathrm{succ}(\nu)$ is continuous for the Scott topology, since it
  preserves suprema.

  Now suppose that $\beta$ is a limit ordinal.
  Then $\beta$ is a filtered union of
  $\mathrm{succ}(\mu)\subseteq\beta$, for $\mu\in\beta$. The cocone
  $\mathrm{succ}(\mu)\hookrightarrow\beta$ exhibits $\beta$ as a colimit of the
  $\mathrm{succ}(\mu)$, and we can make this a colimit in $\Topo$ by equipping
  $\beta$ with the colimit topology induced by the Scott topology of the continuous
  lattices $\mathrm{succ}(\mu)$: a subset $U\subseteq\beta$ is
  open if $U\cap\mathrm{succ}(\mu)$ is open in $\mathrm{succ}(\mu)$, for all
  $\mu\in\beta$. More explicitly, $U\subseteq\beta$
  is open if:
  \begin{enumerate*}
  \item it is up-closed, and;
  \item for any bounded $D\subseteq\beta$, $\sup D\in U$ implies
    $U\cap D\neq\emptyset$.
  \end{enumerate*}
With this topology,
  $\mathrm{succ}(\mu)\hookrightarrow\beta$ is an embedding of spaces.

  Now assume further that $\beta$ is a regular cardinal, so $(\beta,\leq)$ is a
  $\beta$-filtered ordered set. By the above paragraph, $\beta$ is a
  $\beta$-filtered colimit of continuous
  lattices. But $\beta$ is not a continuous lattice, as it is not complete: it
  lacks a top element.
\end{ex}

\appendix

\section{Lax idempotent 2-monads}
\label{sec:lax-idempotent-2}

A 2-monad $\mathsf{T}=(T,i,m)$ is \emph{lax idempotent} if one of the
following equivalent conditions holds:
\begin{enumerate*}
\item
  $Ti\dashv m$ with identity unit;
\item $m\dashv iT$ with identity counit;
\item there
  exists a modification $\delta\colon Ti\Rrightarrow iT$ such that
  $m\cdot\delta=1$ and $\delta\cdot i=1$;
\item the forgetful 2-functor from the
  2-category of $\mathsf{T}$-algebras and lax morphisms is full and faithful on
  morphisms.
\end{enumerate*}
Some of these conditions appear in~\cite{Kock:KZmonads} and~\cite{MR0424896} in
the context of doctrines. See also~\cite{Marmolejo:Doctrines}. A full list of equivalent conditions with the
respective proofs can be found in~\cite{Kelly:Prop-like}.

A 2-comonad $\mathsf{G}$ on $\mathcal K$ is lax idempotent when the
corresponding monad $\mathsf{G}^{\mathrm{op}}$ on the 2-category
$\mathcal{K}^{\mathrm{op}}$, obtained by reversing the morphisms, is lax
idempotent.

\subsection{Algebras and morphisms as monad morphisms}
\label{sec:algebr-morph-as}
There is a standard way of regarding algebras and algebra morphisms as monad
maps, introduced in~\cite{MR0364394}, that we proceed to describe.

Assume that $\C$ is a locally small complete \V-category;
strictly speaking, all we need is finite limits and cotensor products
For each
pair of objects $A$ and $B$ denote by $\langle A,B\rangle$ the right Kan
extension of the \V-functor $A\colon \mathbf{1}\to\C$ along $B\colon \mathbf{1}\to\C$.
\begin{equation}
  \label{eq:48}
  \langle A,B\rangle=\Ran_AB=\{\C(-,A),B\}
  \colon\C\longrightarrow\C
\end{equation}
In other words, $\langle A,B\rangle (X)=\{\C(X,A),B\}$, the cotensor product of
$B\in\C$ by the category $\C(X,A)$. There is a morphism
\begin{equation}
  \label{eq:57}
  \mathrm{ev}_{A,B}\colon \langle A,B\rangle (A)=\{\C(A,A),B\} \longrightarrow B
\end{equation}
that is the counit of the cotensor product in \C, ie the morphism picked out by
\begin{equation}
  \label{eq:69}
  \mathbf{1}\xrightarrow{1_A}
  \C(A,A)\xrightarrow{\eta}\C(\{C(A,A),B\},B)
\end{equation}
where $\eta$ is the unit of the cotensor product.
The morphism~\eqref{eq:57}
satisfies the following universal property: for any \V-functor
$S\colon \C\to \C$, the \CAT-functor
\begin{equation}
  [\C,\C](S,\langle A,B\rangle)\xrightarrow{\mathrm{proj}_A} \C(S(A),\langle A,B\rangle(A))
  \xrightarrow{\eqref{eq:57}} \C(S(A),B)
\end{equation}
is an isomorphism.
The assignment $(A,B)\mapsto\langle A,B\rangle$ defines a $\mathbf{CAT}$-functor
$\langle-,-\rangle\colon\C^{\mathrm{op}}\times\C\to\mathrm{End}(\C)$.

As it is
the case of any
right Kan extension of a functor along itself, $\langle A,A\rangle=\Ran_AA$ has a canonical structure of a
\V-monad, sometimes called density monad. More explicitly, the multiplication $\langle A,A\rangle ^2\to\langle
A,A\rangle$ corresponds to the morphism in \C
\begin{equation}
  \langle A,A\rangle^2(A)\xrightarrow{\langle A,A\rangle(\mathrm{ev}_{A,A})}
  \langle A,A\rangle \xrightarrow{\mathrm{ev}_{A,A}}A
\end{equation}
while the unit $1\to\langle A,A\rangle$ corresponds to the identity morphism
$1\colon A\to A$. 

The definition of the \V-monad $\langle A,A\rangle $ is such that the following
statement holds. Suppose that $\mathsf{S}$ is a 2-monad on \C\ and
$\alpha\colon S\Rightarrow\langle A,A\rangle$ is a 2-natural transformation with
corresponding morphism $a\colon S(A)\to A$ in \C. Then, $\alpha$ is a --~strict~--
morphism of 2-monads --~ie, it is compatible with the multiplications and units
in the usual way~-- if and only if $a$ is an $\mathsf{S}$-algebra structure on $A$.

For each morphism $f\colon A\to B$ in $\C$, one may consider the following
comma-object in the \CAT-category $\mathrm{End}(\C)$.
\begin{equation}
  \label{eq:49}
  \xymatrixrowsep{.1cm}
  \diagram
  &\langle A,A\rangle\ar[dr]^{\langle A,f\rangle}&\\
  \langle f,f\rangle_\ell\rrtwocell<\omit>{}\ar[ur]\ar[dr]&&\langle A,B\rangle\\
  &\langle B,B\rangle\ar[ur]_{\langle f,B\rangle}&
  \enddiagram
\end{equation}
Thus, for any endo-2-functor $S$ of \C\, there is an isomorphism between the
category of 2-natural transformations $S\Rightarrow\langle f,f\rangle_\ell$ and
modifications between them, and the comma-category depicted.
\begin{equation}
  \label{eq:58}
  \diagram
  \C(1,f)\downarrow \C(Sf,1)\ar[d]\ar[r]\drtwocell<\omit>{^}
  &\C(S(B),B)\ar[d]^{\C(Sf,1)}\\
  \C(S(A),A)\ar[r]_-{\C(1,f)}&
  \C(S(A),B)
  \enddiagram
\end{equation}
In particular, each 2-natural transformation $S\Rightarrow\langle f,f\rangle_\ell$
is defined by a unique triple formed by two morphisms $a\colon S(A)\to A$ and
$b\colon S(B)\to B$ and a 2-cell $\varphi\colon f\cdot a\Rightarrow b\cdot
Sf$.

The 2-functor $\langle f,f\rangle_\ell$ has a 2-monad structure that makes the
projections
$\langle A,A\rangle\leftarrow\langle f,f\rangle_\ell\to \langle B,B\rangle$
--~strict~-- morphisms of 2-monads. Furthermore, the bijection mentioned in the
previous paragraph restricts to a bijection between morphisms of 2-monads
$\mathsf{S}\to\langle f,f\rangle_\ell$ and triples $(a,b,\varphi)$ where $a\colon
S(A)\to A$ and $b\colon S(B)\to B$ are $\mathsf{S}$-algebra structures, and
$(f,\varphi)$ is a lax morphism of $\mathsf{S}$-algebras from $(A,a)$ to $(B,b)$.
This method of
describing morphisms as monad maps can be found in the literature as early as in~\cite[\S
3]{MR0364394}.

\subsection{Reflections of monads through a functor}
\label{sec:refl-monads-along}

The following lemma can be applied to the situation of a right adjoint functor
$W\colon\C\to\mathcal{B}$.
\begin{lemma}
  \label{l:11}
  Suppose that $W\colon \C\to\mathcal{B}$ is a continuous \V-functor
  where $\C$ is a complete \V-category and $\mathcal{B}$ admits all right Kan
  extensions along $W$ (ie, all limits weighted by $\mathcal{B}(1,W)$).
\begin{enumerate}
\item\label{item:7}
  If $A$, $B$ are objects of \C, then the canonical \V-natural transformation
  \begin{equation}
    \Ran_W(W\langle A,B\rangle)
    \longrightarrow
    \Ran_{W(A)}(W(B))
    \label{eq:50}
  \end{equation}
  is an isomorphism.
\item\label{item:8}
If $f$, $g$ are
  morphisms in $\C$, then the canonical \V-natural transformation
  \begin{equation}
    \label{eq:51}
    \Ran_W(W\langle f,g\rangle_\ell)
    \longrightarrow
    \langle Wf,Wg\rangle_\ell
  \end{equation}
  is an isomorphism.
\end{enumerate}
\end{lemma}
\begin{proof}
  The proof is straightforward if one takes care of showing at each stage the
  existence of necessary limits in $\mathcal{B}$.
  The continuous \V-functor $W$ preserves the cotensor products $\langle
  A,B\rangle(X)=\{\C(X,A),B\}$,  so $W\langle A,B\rangle\cong
  \Ran_A(W(B))$. Thus, there is an isomorphism
  \begin{equation}
    \tau\colon
    \Ran_W(W\langle A,B\rangle)\cong\Ran_W\Ran_A(W(B))\cong\Ran_{W(A)}(W(B))
    =\langle W(A),W(B)\rangle
  \end{equation}
  where the first object exists by the hypothesis on $\mathcal{B}$, the second
  exists by the reason mentioned above, and the third does by iterated Kan
  extensions~\cite[\S 4.4]{Kelly:BCECTrep}

  To prove the second part of the statement, suppose given $f\colon A\to B$ and
  $g\colon C\to D$. It is easy to
  see that the isomorphism of the first part of the statement is natural
  in $A$ and $B$, so we have a commutative diagram as follows.
  \begin{equation}
    \label{eq:63}
    \xymatrixcolsep{1.6cm}
    \xymatrixrowsep{.5cm}
    \diagram
    \Ran_W(W\langle B,C\rangle)\ar[d]_\cong \ar[r]^-{\Ran_WW\langle 1,g\rangle}&
    \Ran_W(W\langle B,D\rangle)\ar[d]_\cong&
    \Ran_W(W\langle A,D\rangle) \ar[l]_-{\Ran_WW\langle f,1\rangle}
    \ar[d]^\cong
    \\
    \langle W(B),W(C)\rangle \ar[r]^-{\langle 1,Wg\rangle}&
    \langle W(B),W(D)\rangle &
    \langle W(A),W(D)\rangle\ar[l]_-{\langle Wf,1\rangle}
    \enddiagram
  \end{equation}
We shall show that $\langle Wf,Wg\rangle_\ell$ exists and is
isomorphic, in a canonical way, to $\Ran_W(W\langle f,g\rangle_\ell)$ in three stages represented by
the following three isomorphisms.
  \begin{gather}
    \Ran_W(W\langle f,g\rangle_\ell)=
    \Ran_W(W(\langle f,1\rangle\downarrow\langle 1,g\rangle))
    \longrightarrow
    \Ran_W \bigl( (W\langle f,1\rangle)\downarrow (W\langle 1,g\rangle) \bigr)
    \\
    \Ran_W \bigl( (W\langle f,1\rangle)\downarrow (W\langle 1,g\rangle) \bigr)
    \longrightarrow
    (\Ran_W W\langle f,1\rangle)\downarrow (\Ran_W W\langle 1,g\rangle)
    \label{eq:64}
    \\
    (\Ran_W W\langle f,1\rangle)\downarrow (\Ran_W W\langle 1,g\rangle)
    \xrightarrow{\phantom{M}\cong\phantom{M}}
    \langle Wf,1\rangle\downarrow\langle 1,Wg\rangle =
    \langle Wf,Wg\rangle_\ell
    \label{eq:61}
  \end{gather}
  By hypothesis, $W$ is continuous, and, in particular, it preserves
  comma-objects, so the first of the above isomorphisms, and in particular, its
  codomain, exists.

  Let us now show
  that the codomain of the second morphism, shown
  in~\eqref{eq:64}, exists and is isomorphic to the domain. In doing so, we
  may substitute $W\langle f,1\rangle$ and $W\langle 1,g\rangle$ by two arbitrary
  \V-functors $F,G\colon \C\to\mathcal{B}$. By hypothesis, the
  $\V$-functor $[\mathcal{B},\mathcal{B}]\to[\C,\mathcal{B}]$ given by
  restricting along $W$ has a right adjoint, namely $\Ran_W$. The \V-functor
  $F\downarrow G\colon\C\to\mathcal{B}$ exists and is constructed pointwise by
  taking comma-objects in $\mathcal{B}$. This is preserved by right adjoint
  $\Ran_W$, so $\Ran_WF\downarrow\Ran_WG$ exists and is canonically isomorphic
  to $\Ran_W(F\downarrow G)$.

  The third transformation~\eqref{eq:61} is the one induced by the
  diagram~\eqref{eq:63}. This completes the proof of the existence of
  $\Ran_W(W\langle f,g\rangle_\ell)\cong \langle Wf,Wg\rangle_\ell$.
\end{proof}

We shall denote the category of \V-monads on \C\ by
$\mathbf{Mnd}(\C)$.
This is the category of monoids in the monoidal category of endo \V-functors of
\C\ and monoid morphisms, by which we mean \V-natural transformations $T\Rightarrow S$ that are
compatible with the units and multiplications --~morphism of \V-monads on \C.

If $W\colon\C\to\mathcal B$ is a \V-functor and $\mathsf{T}=(T,i,m)$ a
\V-monad on $\C$, then $\Ran_W(WT)$ has a canonical structure of a
\V-monad on $\mathcal{B}$, whenever this right Kan extension exists. The unit
is the \V-natural transformation that corresponds to
$Wi\colon W\Rightarrow WT$, while the multiplication corresponds to the
\V-natural transformation
\begin{equation}
  \label{eq:42}
  \Ran_W(WT)\Ran_W(WT)W\longrightarrow \Ran_W(WT)WT\longrightarrow
  WTT\xrightarrow{Wm}WT.
\end{equation}
If $\Ran_W$ always exists, then $E\mapsto\Ran_W(WE)$ is a (lax) monoidal functor
from $\mathrm{End}(\C)$ to $\mathrm{End}(\mathcal B)$, so it induces a
functor
$\mathbf{Mnd}(\C)\longrightarrow \mathbf{Mnd}(\mathcal B)$.
This is the case, for example, when $W$ has a left adjoint.

In a moment we will need the more general notion of morphism
between monads on different \V-categories. If
$(\mathcal{C},\mathsf{T})$ is a \V-monad on \C\ and $(\mathcal{B},\mathsf{S})$ a
\V-monad on $\mathcal{B}$, a morphism
$(\mathcal{C},\mathsf{T})\to(\mathcal{B},\mathsf{S})$ is a \V-functor
$W\colon\C\to\mathcal{B}$ equipped with a \V-natural transformation
$\omega\colon SW\Rightarrow WT$ satisfying compatibility axioms with the
unit and multiplication of the respective monads, that can be found, for
example, in~\cite{street.ftm}.
There is a bijection between \V-natural transformations $\omega$
that make $(W,\omega)$ a morphism of monads
$(\mathcal{C},\mathsf{T})\to(\mathcal{B},\mathsf{S})$ and liftings of $W$ to the
category of algebras, ie \V-functors, depicted by a dashed arrow, that make the
square commutative.
\begin{equation}
  \label{eq:60}
  \xymatrixrowsep{.4cm}
  \diagram
  \mathsf{T}\text-\mathrm{Alg}\ar@{-->}[r]\ar[d]_{V^{\mathsf{T}}}&
  \mathsf{S}\text-\mathrm{Alg}\ar[d]^{V^{\mathsf{S}}}\\
  \C\ar[r]^-W&
  \mathcal B
  \enddiagram
\end{equation}

There is a bijection between morphisms of \V-monads
$(\C,\mathsf{T}) \to(\mathcal B,\mathsf{S})$ over
$W\colon\C\to\mathcal B$ and morphisms $\mathsf{S}\to\Ran_W(WT)$ in
$\mathbf{Mnd}(\C)$. Indeed, a morphism
structure $\omega\colon SW\Rightarrow WT$ on $W$ corresponds to a monoid morphism
$\alpha\colon \mathsf{S}\to\Ran_W(WT)$ via the universal property of $\Ran_W$.
\begin{equation}
  \xymatrixcolsep{2.5cm}
  \diagram
  \mathcal{C}\ar[d]_T\ar[r]^-W\dtwocell<\omit>{<-4>}&
  \mathcal{B}\ar[d]_{\Ran_WWT}\duppertwocell^{S}{\alpha}\\
  \mathcal{C}\ar[r]_-W&
  \mathcal{B}
  \enddiagram
  \quad=\quad
  \diagram
  \mathcal{C}\ar[d]_T\ar[r]^-W\dtwocell<\omit>{<-7>\omega}&
  \mathcal{B}\duppertwocell^{S}{\omit}\\
  \mathcal{C}\ar[r]_-W&
  \mathcal{B}
  \enddiagram
\end{equation}

\begin{lemma}
  \label{l:13}
  Consider a commutative square of \V-functors~\eqref{eq:60}, where $\mathsf{T}$
  and $\mathsf{S}$ are \V-monads on $\C$ and $\mathcal B$, respectively, and
  assume that $\Ran_W$ of \V-functors into $\mathcal{B}$ always exists. Let
  $(W,w)\colon(\mathcal{C},\mathsf{T})\to(\mathcal{B},\mathsf{S})$ be the
  morphism of \V-monads associated to the square~\eqref{eq:60} and
  $\alpha\colon \mathsf{S}\to\Ran_W(WT)$ the morphism in
  $\mathbf{Mnd}(\mathcal{B})$ associated to $(W,\omega)$. If the
  square~\eqref{eq:60} is a pullback, then $(W,w)$ exhibits $\mathsf{T}$ as a
  reflection of $\mathsf{S}$ along
  $\Ran_W(W-)\colon\mathbf{Mnd}(\mathcal{C})\to\mathbf{Mnd}(\mathcal{B})$.
\end{lemma}
\begin{proof}
  Let $\mathsf{P}$ be a \V-monad on $\C$ and consider the following sets:
  \begin{enumerate}[label=(\alph*)]
  \item \label{item:1} $\mathbf{Mnd}(\C)(\mathsf{T},\mathsf{P})$;
  \item \label{item:3} \V-functors
    $\mathsf{P}\text-\mathrm{Alg}\to\mathsf{T}\text-\mathrm{Alg}$ that
    commute with the respective forgetful functors;
  \item \label{item:4} $\mathbf{Mnd}(\mathcal B)(\mathsf{S},\Ran_W(WP))$;
  \item \label{item:6} \V-functors
    $\mathsf{P}\text-\mathrm{Alg}\to\mathsf{S}\text-\mathrm{Alg}$ over $W$.
  \end{enumerate}
  The sets \ref*{item:1} and \ref*{item:3} are bijective, as well as the sets \ref*{item:4} and
  \ref*{item:6}, and the bijections are natural in $\mathsf{P}$, by the comments
  previous to this lemma. To say that $\mathsf{T}$ is the reflection of
  $\mathsf{S}$ is equally saying that \ref*{item:1} is naturally bijective to
  \ref*{item:4}, while to say that \eqref{eq:60}~is a pullback implies that
  \ref*{item:3} and \ref*{item:6} are naturally bijective.
\end{proof}
  The existence of $\Ran_W$ in Lemma~\ref{l:13} is satisfied,
  for example, when $W$ has a left adjoint.

\subsection{Reflections of lax idempotent \V-monads}
\label{sec:refl-lax-idemp}

  There is a characterisation of lax idempotent 2-monads, due
  to~\cite{Kelly:Prop-like}, which we will find useful. Given a morphism  $f\colon A\to B$ in \C,
  we denote by $\sigma_f\colon\langle
  f,f\rangle_\ell\to\langle A,A\rangle\times\langle B,B\rangle$ the morphism of
  \V-monads induced by the projections of the comma-object~\eqref{eq:49}. A \V-monad
  $\mathsf{T}$ is lax idempotent precisely when it is co-orthogonal to each $\sigma_f$
  in $\mathbf{Mnd}(\C)$, ie when $\mathbf{Mnd}(\C)(\mathsf{T},\sigma_f)$ is a
  bijection. 

\begin{cor}
  \label{cor:9}
  If the \V-monad $\mathsf{S}$ in Lemma~\ref{l:13} is lax idempotent,
  $\C$ is cocomplete, $\mathcal B$ admits $\Ran_W$ and $W$ is continuous, then
  $\mathsf{T}$ is lax idempotent.
\end{cor}
\begin{proof}
  By the remarks previous to the present corollary, we must show that $\mathsf{T}$ is co-orthogonal to the morphism
  $\sigma_f\colon\langle f,f\rangle_\ell\to\langle A,A\rangle\times\langle
  B,B\rangle$, for any morphism $f\colon A\to B$ in $\C$. Equivalently, that
  $\mathsf{S}$ is co-orthogonal to $\Ran_W(W\sigma_f)$ by Lemma~\ref{l:13}. This
  morphism is isomorphic to $\sigma_{Wf}$, via the
  isomorphisms $\Ran_W(W\langle A,A\rangle)\cong\langle(W(A),W(A)\rangle$ and
  $\Ran_W(W\langle f,f\rangle_\ell)\cong \langle W(f),W(f)\rangle_\ell$ of
  Lemma~\ref{l:11}, from where it is obvious that $\mathsf{S}$ is co-orthogonal
  to $\Ran_W(W\sigma)$.
\end{proof}

\bibliographystyle{abbrv}
\bibliography{biblio.bib}
\end{document}

%% file: example.tex
Let $(\mathsf{E},\mathsf{M})$ denote the \textsc{lofs} on \Cat\ whose
$\mathsf{M}$-algebras are cloven Grothendieck opfibrations. A pair of morphisms
$f\colon A\to B$ and $g\colon C\to D$ in a 2-category \C\ satisfies
$f\pitchfork_{\mathsf{M}}g$ when the comparison functor
$\C(B,C)\to\C(A,C)\times_{\C(A,D)}\C(B,D)$ is an opfibration. This can be
unpacked in the following way. Given a diagram
\begin{equation}
  \label{eq:2}
  \diagram
  A\ar[d]_f\ar[r]^h\rtwocell<\omit>{^<1.5>\alpha}&
  C\ar[d]^g\\
  B\ar[ur]|d\ar[r]_k\rtwocell<\omit>{<-1.5>\beta}&
  D
  \enddiagram
\end{equation}
with $\beta\cdot f=g\cdot \alpha$, there exists a pre-opcartesian lifting, which is
a 2-cell $\gamma\colon d\Rightarrow \bar d$ such that $\gamma\cdot f=\alpha$ and
$g\cdot\gamma=\beta$. Its pre-opcartesian property means that any other
$\gamma'\colon d\Rightarrow d'$ with the same property as $\gamma$ factors
uniquely through $\gamma$. Furthermore, pre-opcartesian morphisms should be closed
under composition.

Therefore, any small set of functors $\{f_i\}$ cofibrantly
$(\mathsf{E},\mathsf{M})$-generates an \textsc{awfs} (in fact, a \textsc{lofs},
as we shall see in \S \ref{sec:lax-orthogonality-kz})
$(\mathsf{L},\mathsf{R})$ on \Cat. The $\mathsf{R}$-algebras are
those functors $g\colon C\to D$ with the extra structure described in the
previous paragraph, for each $f_i$.

Let us look at the example of the family with one element $0\colon
\mathbf{1}\to\two$, the functor that picks out $0\in\two=(0\to 1)$. This functor
plays the role of $f$ above. A functor $h\colon \mathbf{1}\to C$ is an object
$c\in C$, $k$ and $d$ are morphisms in $D$ and $C$, respectively. The
transformations $\alpha$ and $\beta$ are a morphism $\dom(d)\to c$ and a
commutative square as depicted on the right below, respectively.
\begin{equation}
  \label{eq:2}
  \diagram
  \bullet \ar[r]^d\ar[d]_\alpha&\bullet\ar@{..>}[d]\\
  \bullet\ar@{..>}[r]&\circ
  \enddiagram
  \quad
  \xrightarrow{\phantom{M}g\phantom{M}}
  \quad
  \diagram
  \bullet\ar[d]_{g(\alpha)}\ar[r]^{g(d)}&\bullet\ar[d]^{\beta_1}\\
  \bullet\ar[r]^-{k}&\bullet
  \enddiagram
\end{equation}
The universal property of $d$ asserts that the solid cospan on the left can be
completed to a commutative square that is mapped by $g$ to the square on the
right.
Furthermore, this dotted completion is universal,
translating the pre-opcartesian property. This means that any other completion of the
solid cospan to a commutative square that is mapped by $g$ to the square on the
right factors uniquely by the universal square. The fact that pre-opcartesian
morphisms are closed under composition is, in this case, automatic, as it means
that pushout squares are closed under pasting.

We may specialise to the case of an $\mathsf{R}$-algebra of the form
$g\colon C\to \mathbf{1}$, in other words, to fibrant objects. In this case all we have is a choice of a pushout for
each cospan in $C$. In other words, the fibrant replacement monad of
$(\mathsf{L},\mathsf{R})$ is the monad on \Cat\ whose algebras are categories
with chosen pullbacks. A general $\mathsf{R}$-algebra can be regarded as a
functor that has pushouts fibrewise.


%% file: multicat2.tex
As an application of the theory developed thus far, in this section we exhibit
opfibrations of multicategories as the right part of a \textsc{lofs} on the
2-category of multicategories that is cofibrantly \kz-generated by a double
category of easy description. Opfibrations of multicategories were considered
(and called covariant fibrations) in \cite{MR2075589} as generalisations of the
representable multicategories, in the sense that the latter are opfibrations
over the terminal multicategory.

\subsection{Background on multicategories}
\label{sec:backgr-mult}

We assume the reader is acquainted with the notion of multicategory, or
non-$\Sigma$ coloured operad. The idea goes back to
J.~Lambek~\cite{Lambek:DeductiveII} and modern expositions can be found in
\cite{Hermida:RepresMulticats} and \cite[I.2]{Leinster:HOHC}.

We shall represent multicategories by blackboard bold letters, so $\mathbb{A}$
will be a multicategory. The 2-category of multicategories will be \multicat. A \emph{module}
$\phi\colon \mathbb{A}\nrightarrow \mathbb{B}$ consists of sets
$\phi(b_1,\dots,b_n;a)$ with actions of $\mathbb{A}$ and $\mathbb{B}$ on each
side. See, for example, \cite[I.2.3]{Leinster:HOHC}. The \emph{collage} for
$\phi$ is the multicategory
$\operatorname{coll}(\phi)$ with objects the disjoint union of those of
$\mathbb{A}$ and $\mathbb{B}$, which contains both multicategories as full
sub-multicategories, and with
$\operatorname{coll}(\phi)(b_1,\dots,b_n,a)=\phi(b_1,\dots,b_n;a)$. The
remainder multihoms are empty. The
composition of $\operatorname{coll}(\phi)$ uses the module structure of
$\phi$. We shall denote by
$j_{\mathbb{B}}\colon \mathbb{B}\to \operatorname{coll}(\phi)$ the full
inclusion. Collages enjoy a universal property that we shall need only in the
following special case.

The $n$th
cardinal $\{0<1<\dots<n-1\}$ we shall denote simply by $n$, and regard it
both as a set and as a discrete multicategory. Consider the module
$\phi_n\colon 1\nrightarrow n$ given by $\phi_n(0,1,\dots,n-1;0)=1$ and all the
other possibilities equal to $\emptyset$. Then
$C(n)\coloneq\operatorname{coll}(\phi_n)$ has objects $0,\dots,n-1$ and an extra object
that we denote by $*$. Aside from the multihoms that correspond to the identity
maps, there is a single non-empty multihom
$C(n)(0,\dots,n-1;*)=1$. There clearly is a bijection between morphisms of
multicategories $C(n)\to\mathbb{B}$ and multimorphisms of the form
$b_0,\dots,b_{n-1}\to b$ in $\mathbb{B}$.

\subsection{Opfibrations of multicategories}
\label{sec:opfibr-mult}

Let us now turn to opfibrations of multicategories, called covariant fibrations
of multicategories in~\cite{MR2075589}. Given a multicategory morphism $P\colon
\mathbb{E}\to\mathbb{B}$ and a multimap $f\colon P(e_1),\dots, P(e_n)\to b$ in
$\mathbb{B}$, an opcartesian lifting of $f$ is a multimap $\bar f\colon
e_1,\dots, e_n\to e$ in $\mathbb{E}$ with the property that: for any $g\colon
e_1,\dots,e_n\to x$ and $w\colon b\to P(x)$ such that $P(g)=w\cdot f$,
there exists a unique $v\colon e\to x$ satisfying $v\cdot f=g$.
One says that $P$ is an \emph{opfibration of
multicategories} if each multimap has an opcartesian lifting. This definition
is completely analogous to that of an opfibration of categories, only with the
appropriate modifications to allow for multimaps. The usual argument shows that the
definition we use is equivalent to that of~\cite{MR2075589}.

A \emph{cloven} opfibration of multicategories is an opfibration with extra
\emph{structure,} a cleavage, that provides a choice of opcartesian lifting for
each multimap $f$ as in the previous paragraph. It is easy to verify that to give a
cleavage for $P$ is the same as giving, for each $n$, a \textsc{rali} structure
to the functor
\begin{equation}
  \multicat(C(n),\mathbb{E})\to
  \mathcal{E}^n\times_{\mathbb{B}^n}\multicat(C(n),\mathbb{B})
  \label{eq:6}
\end{equation}
(where $\mathcal{E}$ is the underlying category of $\mathbb{E}$) whose first
coordinate is induced by pre-composing with $n\hookrightarrow C(n)$, while the
second coordinate is induced by post-composing with $P$.

If $P$ and $P'$ are two cloven opfibrations of multicategories, a strict
morphism $P\to P'$ is a morphism of $\multicat^\two$
\begin{equation}
  \label{eq:164}
  \xymatrixrowsep{.3cm}
  \diagram
  \mathbb{E}\ar[d]_P\ar[r]^-{H}&
  \mathbb{E}'\ar[d]^{P'}\\
  \mathbb{B}\ar[r]^-K&\mathbb{B}'
  \enddiagram
\end{equation}
where $H$ preserves the chosen opcartesian liftings. This means that if $f$ is
an opcartesian multimap in $\mathbb{E}$, then $H(f)$ is the chosen opcartesian
lifting of $K(f)$. In this way we obtain a category $\clopfib$ of cloven
opfibrations of multicategories, together with a forgetful functor
\begin{equation}
  \clopfib\longrightarrow \multicat^\two.\label{eq:165}
\end{equation}
We make \clopfib\ into a 2-category, and \eqref{eq:165} into a 2-functor, by
declaring that~\eqref{eq:165} is locally full and faithful.

Write $\arepmult$ for the fibre of $\clopfib$ \eqref{eq:165} over
$\mathbf{1}$. Its objects we call \emph{algebraic representable
  multicategories}.
A slight modification of the arguments of~\cite{Hermida:RepresMulticats} shows
that an algebraic representable multicategory structure on $\mathbb{A}$ amounts
to an unbiased monoidal category structure on the underlying category of
$\mathbb{A}$ --~see \cite[I.3.1]{Leinster:HOHC}. There
is an isomorphism between $\arepmult$ and the 2-category
$\mathbf{UMonCat}_s$ of unbiased monoidal categories and strict monoidal
functors.

The last piece of background we shall need is the fact that $\multicat$ is
locally finitely presentable, as an ordinary category. This can be shown
directly, by hand, as it were. It will be more convenient to appeal to \cite[D.1]{Leinster:HOHC},
specifically Theorem~6.5.4 and Proposition~6.5.6. These deal with generalised
multicategories, whose type, or shape, is defined by a certain monad $T$ on a
category $\mathcal{E}$. In our case, these are the free monoid monad on
$\mathbf{Set}$. The said results tell us that the forgetful functor from
$\multicat$ to multigraphs is monadic with a finitary associated monad. The
category of multigraphs is locally finitely presentable, being the slice of
$\mathbf{Set}$ over
\begin{equation}
  \label{eq:166}
  \mathbf{Set}\xrightarrow{\Delta}
  \mathbf{Set}\times\mathbf{Set}
  \xrightarrow{T\times 1}
  \mathbf{Set}\times\mathbf{Set}
  \xrightarrow{\times}\mathbf{Set}
\end{equation}
where $T$ is the free monoid monad. We deduce that $\multicat$ is itself locally
finitely presentable.

\subsection{Cofibrant \kz-generation}
\label{sec:cofibrant-generation-2}

Let $\mathcal{F}$ be the set of finite cardinals (ie, the set of natural
numbers) and $U\colon \mathcal{F}\to \multicat^\two$ the functor given by
$U(n)=j_n\colon n\hookrightarrow C(n)$.
An object of $\mathcal{F}^{\pitchfork_{\mathkz}}$ consists of a morphism of
multicategories $P\colon \mathbb{E}\to\mathbb{B}$ with a \textsc{rali} structure
on the comparison functor
\begin{equation}
  \multicat(C(n),\mathbb{E})\to
  \multicat(n,\mathbb{E})\times_{\multicat(n,\mathbb{B})}\multicat(C(n),\mathbb{B})
  \label{eq:7}
\end{equation}
for each $n$. This is the same functor as \eqref{eq:6}, so we see that
$\mathcal{F}^{\pitchfork_{\mathkz}}$ has cloven opfibrations as objects. One
easily verifies that its morphisms are morphisms of opfibrations that preserve
cleavages.
\begin{prop}
  \label{prop:3}
  \begin{enumerate*}
  \item There is an isomorphism
    $\mathbf{OpFib}\cong\mathcal{F}^{\pitchfork_{\mathkz}}$.
  \item There is an isomorphism between the fibre of
    $\mathcal{F}^{\pitchfork_{\mathkz}}$ over $\mathbf{1}$ and the 2-category of
    unbiased monoidal categories and strict monoidal functors.
  \end{enumerate*}
\end{prop}

An application of Corollary~\ref{cor:6} and Theorem~\ref{thm:1} yields:
\begin{prop}
  \label{prop:11}
  The category $\mathcal{F}$ over $\multicat^\two$ cofibrantly \slkz-generates a
  {\normalfont\textsl{\textsc{lofs}}} $(\mathsf{L},\mathsf{R})$ on
  \multicat. The $\mathsf{R}$-algebras are the cloven opfibrations of
  multicategories.
\end{prop}

Taking the fibre of $\mathsf{R}\text-\mathrm{Alg}$ over the terminal
multicategory, we easily obtain the following consequence.
\begin{cor}
  \label{cor:1}
  The 2-functor from $\mathbf{UMonCat}$ to $\multicat$ is strictly monadic. The
  associated monad is lax idempotent.
\end{cor}
The corollary can be regarded as a version of
the results in \cite{Hermida:RepresMulticats}. There are a couple of
differences, though. First, the
techniques in \emph{ibid} are suitable for application to generalised
multicategories, as done in \cite{Leinster:HOHC} and \cite{MR2075589}. Secondly,
\cite{Hermida:RepresMulticats} compares strict monoidal categories with
multicategories, while our corollary involves unbiased monoidal categories.

\begin{rmk}
  \label{rmk:1}
  It is not hard to cofibrantly \kz-generate two \textsc{lofs}s on
  \multicat\ whose fibrant objects are normal unbiased monoidal categories,
  and strict monoidal categories, respectively. In order to do this, one
  should add some generating squares to $\mathcal{F}$, and employ the resulting
  double category.
\end{rmk}


%% file: cofgenlaxorth.bbl
\def\cprime{$'$} \def\cprime{$'$}
\begin{bibdiv}
\begin{biblist}

\bib{MR1916156}{article}{
      author={Ad{\'a}mek, Ji{\v{r}}{\'{\i}}},
      author={Herrlich, Horst},
      author={Rosick{\'y}, Walter, Ji{\v{r}}{\'{\i}}and~Tholen},
       title={Weak factorization systems and topological functors},
        date={2002},
        ISSN={0927-2852},
     journal={Appl. Categ. Structures},
      volume={10},
      number={3},
       pages={237\ndash 249},
         url={http://dx.doi.org/10.1023/A:1015270120061},
        note={Papers in honour of the seventieth birthday of Professor Heinrich
  Kleisli (Fribourg, 2000)},
      review={\MR{1916156 (2003i:18001)}},
}

\bib{MR3283679}{article}{
      author={Ad{\'a}mek, Ji{\v{r}}{\'{\i}}},
      author={Sousa, Lurdes},
      author={Velebil, Ji{\v{r}}{\'{\i}}},
       title={Kan injectivity in order-enriched categories},
        date={2015},
        ISSN={0960-1295},
     journal={Math. Structures Comput. Sci.},
      volume={25},
      number={1},
       pages={6\ndash 45},
         url={http://dx.doi.org/10.1017/S0960129514000024},
      review={\MR{3283679}},
}

\bib{BKP}{article}{
      author={Blackwell, R.},
      author={Kelly, G.~M.},
      author={Power, A.~J.},
       title={Two-dimensional monad theory},
        date={1989},
        ISSN={0022-4049},
     journal={J. Pure Appl. Algebra},
      volume={59},
      number={1},
       pages={1\ndash 41},
}

\bib{MR3393453}{article}{
      author={Bourke, John},
      author={Garner, Richard},
       title={Algebraic weak factorisation systems {I}: {A}ccessible {AWFS}},
        date={2016},
        ISSN={0022-4049},
     journal={J. Pure Appl. Algebra},
      volume={220},
      number={1},
       pages={108\ndash 147},
         url={http://dx.doi.org/10.1016/j.jpaa.2015.06.002},
      review={\MR{3393453}},
}

\bib{MR2927175}{article}{
      author={Cagliari, Francesca},
      author={Clementino, Maria~Manuel},
      author={Mantovani, Sandra},
       title={Fibrewise injectivity and {K}ock-{Z}\"oberlein monads},
        date={2012},
        ISSN={0022-4049},
     journal={J. Pure Appl. Algebra},
      volume={216},
      number={11},
       pages={2411\ndash 2424},
         url={http://dx.doi.org/10.1016/j.jpaa.2012.03.019},
      review={\MR{2927175}},
}

\bib{Clementino2016458}{article}{
      author={Clementino, Maria~Manuel},
      author={{L\'opez Franco}, Ignacio},
       title={Lax orthogonal factorisation systems},
        date={2016},
        ISSN={0001-8708},
     journal={Advances in Mathematics},
      volume={302},
       pages={458 \ndash  528},
  url={http://www.sciencedirect.com/science/article/pii/S0001870816309604},
}

\bib{lmcs3960}{article}{
      author={Clementino, Maria~Manuel},
      author={L\'opez~Franco, Ignacio},
       title={Lax orthogonal factorisations in monad-quantale-enriched
  categories},
        date={2017},
     journal={Logical Methods in Computer Science},
      volume={13},
}

\bib{1702.02602}{misc}{
      author={Clementino, Maria~Manuel},
      author={{L\'opez Franco}, Ignacio},
       title={Lax orthogonal factorisations in ordered structures},
        date={2017},
        note={arXiv:1702.02602},
}

\bib{MR0367013}{article}{
      author={Day, Alan},
       title={Filter monads, continuous lattices and closure systems},
        date={1975},
        ISSN={0008-414X},
     journal={Canad. J. Math.},
      volume={27},
       pages={50\ndash 59},
      review={\MR{0367013 (51 \#3258)}},
}

\bib{Dubuc.Kan}{book}{
      author={Dubuc, Eduardo~J.},
       title={{Kan extensions in enriched category theory}},
   publisher={{Lecture Notes in Mathematics. 145. Berlin-Heidelberg-New York:
  Springer-Verlag. XVI, 173 p.}},
        date={1970},
}

\bib{MR1641443}{article}{
      author={Escard{\'o}, Mart{\'{\i}}n~H{\"o}tzel},
       title={Properly injective spaces and function spaces},
        date={1998},
        ISSN={0166-8641},
     journal={Topology Appl.},
      volume={89},
      number={1-2},
       pages={75\ndash 120},
         url={http://dx.doi.org/10.1016/S0166-8641(97)00225-3},
        note={Domain theory},
      review={\MR{1641443 (2000a:06035)}},
}

\bib{MR0322004}{article}{
      author={Freyd, P.~J.},
      author={Kelly, G.~M.},
       title={Categories of continuous functors. {I}},
        date={1972},
        ISSN={0022-4049},
     journal={J. Pure Appl. Algebra},
      volume={2},
       pages={169\ndash 191},
      review={\MR{0322004 (48 \#369)}},
}

\bib{MR2506256}{article}{
      author={Garner, Richard},
       title={Understanding the small object argument},
        date={2009},
        ISSN={0927-2852},
     journal={Appl. Categ. Structures},
      volume={17},
      number={3},
       pages={247\ndash 285},
         url={http://dx.doi.org/10.1007/s10485-008-9137-4},
      review={\MR{2506256 (2010a:18005)}},
}

\bib{MR2283020}{article}{
      author={Grandis, Marco},
      author={Tholen, Walter},
       title={Natural weak factorization systems},
        date={2006},
        ISSN={0044-8753},
     journal={Arch. Math. (Brno)},
      volume={42},
      number={4},
       pages={397\ndash 408},
      review={\MR{2283020 (2008b:18006)}},
}

\bib{MR0213413}{incollection}{
      author={Gray, John~W.},
       title={Fibred and cofibred categories},
        date={1966},
   booktitle={Proc. {C}onf. {C}ategorical {A}lgebra ({L}a {J}olla, {C}alif.,
  1965)},
   publisher={Springer, New York},
       pages={21\ndash 83},
      review={\MR{0213413 (35 \#4277)}},
}

\bib{Hermida:RepresMulticats}{article}{
      author={Hermida, Claudio},
       title={Representable multicategories},
        date={2000},
     journal={Adv. Math.},
      volume={151},
      number={2},
       pages={164\ndash 225},
}

\bib{MR2075589}{incollection}{
      author={Hermida, Claudio},
       title={Fibrations for abstract multicategories},
        date={2004},
   booktitle={Galois theory, {H}opf algebras, and semiabelian categories},
      series={Fields Inst. Commun.},
      volume={43},
   publisher={Amer. Math. Soc., Providence, RI},
       pages={281\ndash 293},
      review={\MR{2075589}},
}

\bib{MR0364394}{incollection}{
      author={Kelly, G.~M.},
       title={Coherence theorems for lax algebras and for distributive laws},
        date={1974},
   booktitle={Category {S}eminar ({P}roc. {S}em., {S}ydney, 1972/1973)},
   publisher={Springer},
     address={Berlin},
       pages={281\ndash 375. Lecture Notes in Math., Vol. 420},
}

\bib{MR589937}{article}{
      author={Kelly, G.~M.},
       title={A unified treatment of transfinite constructions for free
  algebras, free monoids, colimits, associated sheaves, and so on},
        date={1980},
        ISSN={0004-9727},
     journal={Bull. Austral. Math. Soc.},
      volume={22},
      number={1},
       pages={1\ndash 83},
         url={http://dx.doi.org/10.1017/S0004972700006353},
      review={\MR{589937 (82h:18003)}},
}

\bib{Kelly:BCECTrep}{article}{
      author={Kelly, G.~M.},
       title={Basic concepts of enriched category theory},
        date={2005},
        ISSN={1201-561X},
     journal={Repr. Theory Appl. Categ.},
      number={10},
       pages={vi+137 pp. (electronic)},
      review={\MR{MR2177301}},
}

\bib{Kelly:Prop-like}{article}{
      author={Kelly, G.~M.},
      author={Lack, Stephen},
       title={On property-like structures},
        date={1997},
        ISSN={1201-561X},
     journal={Theory Appl. Categ.},
      volume={3},
       pages={No. 9, 213\ndash 250 (electronic)},
}

\bib{elementsof2cat}{incollection}{
      author={Kelly, G.M.},
      author={Street, Ross},
       title={{Review of the elements of 2-categories}},
        date={1974},
   booktitle={{Category Sem., Proc., Sydney 1972/1973, Lect. Notes Math. 420,
  75-103 }},
   publisher={{Springer}},
}

\bib{Kelly:BCECT}{book}{
      author={Kelly, Gregory~Maxwell},
       title={Basic concepts of enriched category theory},
      series={London Mathematical Society Lecture Note Series},
   publisher={Cambridge University Press},
     address={Cambridge},
        date={1982},
      volume={64},
        ISBN={0-521-28702-2},
}

\bib{Kock:KZmonads}{article}{
      author={Kock, Anders},
       title={Monads for which structures are adjoint to units},
        date={1995},
        ISSN={0022-4049},
     journal={J. Pure Appl. Algebra},
      volume={104},
      number={1},
       pages={41\ndash 59},
      review={\MR{MR1359690 (96k:18009)}},
}

\bib{Lambek:DeductiveII}{incollection}{
      author={Lambek, J.},
       title={{Deductive systems and categories. II. Standard constructions and
  closed categories}},
        date={1969},
   booktitle={{Category Theory Homology Theory Appl., Proc. Conf. Seattle Res.
  Center Battelle Mem. Inst. 1968, 1, 76-122 }},
}

\bib{Lawvere:Metricspaces}{misc}{
      author={Lawvere, F.William},
       title={{Metric spaces, generalized logic, and closed categories}},
        date={1974},
}

\bib{Leinster:HOHC}{book}{
      editor={Leinster, Tom},
       title={Higher operads, higher categories},
      series={London Mathematical Society Lecture Note Series},
   publisher={Cambridge University Press},
     address={Cambridge},
        date={2004},
      volume={298},
        ISBN={0-521-53215-9},
      review={\MR{MR2094071}},
}

\bib{Marmolejo:Doctrines}{article}{
      author={Marmolejo, F.},
       title={Doctrines whose structure forms a fully faithful adjoint string},
        date={1997},
     journal={Theory Appl. Categ.},
      volume={3},
       pages={No.\ 2, 24\ndash 44 (electronic)},
}

\bib{MR0223432}{book}{
      author={Quillen, Daniel~G.},
       title={Homotopical algebra},
      series={Lecture Notes in Mathematics, No. 43},
   publisher={Springer-Verlag, Berlin-New York},
        date={1967},
      review={\MR{0223432 (36 \#6480)}},
}

\bib{MR0404073}{incollection}{
      author={Scott, Dana},
       title={Continuous lattices},
        date={1972},
   booktitle={Toposes, algebraic geometry and logic ({C}onf., {D}alhousie
  {U}niv., {H}alifax, {N}. {S}., 1971)},
   publisher={Springer},
     address={Berlin},
       pages={97\ndash 136. Lecture Notes in Math., Vol. 274},
      review={\MR{0404073 (53 \#7879)}},
}

\bib{street.ftm}{article}{
      author={Street, Ross},
       title={{The formal theory of monads}},
        date={1972},
     journal={J. Pure Appl. Algebra},
      volume={2},
       pages={149\ndash 168},
}

\bib{MR0346027}{article}{
      author={Street, Ross},
      author={Walters, R. F.~C.},
       title={The comprehensive factorization of a functor},
        date={1973},
        ISSN={0002-9904},
     journal={Bull. Amer. Math. Soc.},
      volume={79},
       pages={936\ndash 941},
      review={\MR{0346027 (49 \#10753)}},
}

\bib{MR0424896}{article}{
      author={Z{\"o}berlein, Volker},
       title={Doctrines on {$2$}-categories},
        date={1976},
     journal={Math. Z.},
      volume={148},
      number={3},
       pages={267\ndash 279},
}

\end{biblist}
\end{bibdiv}
